\documentclass[11pt]{article}
\usepackage{enumitem}
\usepackage{graphicx,xcolor}
\usepackage{epstopdf}
\usepackage[font=scriptsize,labelfont=bf]{caption}
\usepackage[font=scriptsize]{subcaption}
\usepackage{resizegather}

\usepackage[hyphens]{url}

\usepackage{booktabs} 

\graphicspath{{./figures/}}

\usepackage[onehalfspacing,nodisplayskipstretch]{setspace}

\usepackage{amsfonts} 
\usepackage{bigints}
\usepackage{bm}
\usepackage{amsmath,amssymb,amsbsy,amsthm}

\usepackage{epsfig}

\newcommand {\be}[1]{\begin{equation}\label{#1}}
\newcommand {\ee}{\end{equation}}
\newcommand {\bea}{\begin{eqnarray}}
\newcommand {\eea}{\end{eqnarray}}

\newcommand {\bE}{{\mathbb E}}

\newcommand{\oFs}{\overline{F}^{*}}

\def\texitem#1{\par\smallskip\noindent\hangindent 25pt
               \hbox to 25pt {\hss #1 ~}\ignorespaces}

\newtheorem{theorem}{Theorem}

\newtheorem{prop}{Proposition}

\newcommand{\appsection}[1]{\let\oldthesection\thesection
  \renewcommand{\thesection}{Electronic Companion \oldthesection}
  \section{#1}\let\thesection\oldthesection}

\theoremstyle{remark}

\newcommand\RV{\mathcal{RV}}
\newcommand\prob{\mathbb{P}}

\numberwithin{equation}{section}
\numberwithin{theorem}{section}
\numberwithin{corollary}{section}
\numberwithin{prop}{section}

\graphicspath{{./figs/}}

\definecolor{darkred}{RGB}{139,0,0}
\definecolor{darkgreen}{RGB}{0,139,0}
\newcommand{\DAS}[1]{{\color{black} #1}}

\begin{document}


\title{On the heavy-tail behavior of the distributionally robust newsvendor}
\author{Bikramjit Das\thanks{Engineering Systems and Design, Singapore University of Technology and Design, 8 Somapah Road, Singapore 487372. Email: bikram@sutd.edu.sg} \and Anulekha Dhara\thanks{Industrial and Operations Engineering, University of Michigan, Ann Arbor, Michigan 48109. Email: adhara@umich.edu} \and Karthik Natarajan\thanks{Engineering Systems and Design, Singapore University of Technology and Design, 8 Somapah Road, Singapore 487372. Email: karthik\_natarajan@sutd.edu.sg}
\footnote{The research of the first author was partly supported by the MOE Academic Research Fund Tier 2 grant MOE2017-T2-2-161, ``Learning from common connections in social networks" second and third authors was partly supported by MOE Academic Research Fund Tier 2 grant T2MOE1706, ``On the Interplay of Choice, Robustness and Optimization in Transportation" and the SUTD-MIT International Design Center grant IDG21700101 on ``Design of the
Last Mile Transportation System: What Does the Customer Really Want?".}
}
\date{First submitted: June 2018, Revised: August 2019}
\maketitle


\begin{abstract}
Since the seminal work of Scarf (1958) [A min-max solution of an inventory problem, Studies in the Mathematical Theory of Inventory and Production, pages 201-209] on the newsvendor problem with ambiguity in the demand distribution, there has been a growing interest in the study of the distributionally robust newsvendor problem. The model is criticized at times for being overly conservative since the worst-case distribution is discrete with a few support points. However, it is the order quantity prescribed from the model that is of practical relevance. A simple calculation shows that the optimal order quantity in Scarf's model with known first and second moment is also optimal for a censored student-t distribution with parameter 2. In this paper, we generalize this ``heavy-tail optimality" property of the distributionally robust newsvendor to an ambiguity set where information on the first and the $\alpha$th moment is known, \textcolor{black}{for any real number $\alpha > 1$}. \textcolor{black}{We show that the optimal order quantity for the distributionally robust newsvendor problem is also optimal for a regularly varying distribution with roughly a power law tail with tail index $\alpha$. We illustrate the usefulness of the model in the high service level regime with numerical experiments, by showing that when a standard distribution such as the exponential or lognormal distribution is contaminated with a heavy-tailed (regularly varying) distribution, the distributionally robust optimal order quantity outperforms the optimal order quantity of the original distribution, even with a small amount of contamination.}
\end{abstract}

\maketitle

\section{Introduction}  \label{sec1}
Since the pioneering work of Scarf \cite{scarf58}, there has been a growing interest in the study of the distributionally robust newsvendor problem where the probability distribution of the demand is ambiguous. Formally, the problem is stated as follows. A newsvendor needs to decide on the number of units of an item to order before the actual demand is observed. The unit purchase cost is $c$ where $c > 0$ and the unit revenue is $p$ where $p > c$. Any unsold item at the end of the selling period has zero salvage value. The demand $\tilde{d}$ for the item is random and unknown before the order is placed. Furthermore, the probability distribution of the demand, denoted by $F(w):=\prob(\tilde{d}\le w)$ is ambiguous and only assumed to lie in a set of possible distributions $\mathcal{F}$. The ambiguity in the demand distribution might arise due to one of several reasons. It might arise as a subjective input when a new product is introduced into the market for which past demand data is unavailable and one is unsure about an exact distribution or it might arise when a set of plausible demand distributions is constructed from past data using moments, structural information or probability distance metrics. All relevant information on the demand distribution that the newsvendor possesses is assumed to be captured in the set $\mathcal{F}$. The distributionally robust newsvendor then orders the quantity that maximizes the minimum (worst-case) expected profit. Mathematically, this problem is formulated as choosing an order quantity $q$ that solves:
\begin{equation}
\begin{array}{lll} \label{a}
\displaystyle \max_{q \in \Re_{+}} & \displaystyle \inf_{F \in \mathcal{F}} & \displaystyle \left( p\mathbb{E}_{F}[\min(q,\tilde{d})] - cq\right).
\end{array}
\end{equation}
Using the relation $\min(d,q) = d - [d-q]_+$, where $[d-q]_+ = \max(0,d-q)$, the optimal order quantity in (\ref{a}) is equivalently obtained by solving the problem:
\begin{equation}
\begin{array}{lll} \label{c}
\displaystyle \min_{q \in \Re_{+}} & \displaystyle  \left((1-\eta)q + \sup_{F \in \mathcal{F}}\mathbb{E}_{F}[\tilde{d}-q]_+\right),
\end{array}
\end{equation}
under the assumption that the mean value of demand is specified in the set $\mathcal{F}$, where $\eta = 1-{c}/{p} \in [0,1)$ denotes the critical ratio. The formulation in \eqref{c} can be interpreted as a worst-case expected cost minimization version of the newsvendor problem.

\subsection{Scarf's Model} \label{subsec:scarf}
The earliest version of the model is attributed to Scarf \cite{scarf58} who assumed that the mean and the variance of the demand are specified in the set $\mathcal{F}$, but the exact form of the distribution is unknown. The set of demand distributions is defined as:
\begin{equation*} \label{b}
\mathcal{F}_{1,2} = \left\{F \in {\mathbb M}(\Re_+): \displaystyle\int_{0}^\infty \mathrm d F(w) = 1,~ \displaystyle\int_{0}^\infty w\:\mathrm dF(w) = m_1,~ \displaystyle\int_{0}^\infty w^2\: \mathrm d F(w) = m_2\right\},
\end{equation*}
where $\mathbb{M}(\Re_+)$ is the set of finite positive Borel measures supported on the non-negative real line and $m_1$ and $m_2$ are the first and second moment which are assumed to satisfy $m_2 > m_1^2 > 0$. Note that when $m_2 = m_1^2$, the demand is deterministic with support at $m_1$ and the optimal order quantity is trivially $m_1$. In the standard newsvendor model, when the set of distributions is a singleton, the optimal order quantity reduces to the well-known critical fractile formula $q^* = F^{-1}(\eta)$, where $F^{-1}(\cdot)$ is the generalized inverse of the cumulative distribution function. However, in the  robust model, the worst-case demand distribution might change with the order quantity. Scarf \cite{scarf58} explicitly characterized the unique two point distribution that attains the worst-case in (\ref{c}) for the set $\mathcal{F}_{1,2}$. Given an order quantity $q \geq m_2/2m_1$, the worst-case demand distribution for $\sup_{F \in \mathcal{F}_{1,2}} \:\mathbb{E}_F[\tilde{d} - q]_+$  is given by the distribution with two support points:
\begin{equation*} \label{d0}
\tilde{d}^*_q = \left\{\begin{array}{llr}
\displaystyle q - \sqrt{q^2-2m_1q+m_2}, & \textrm{w.p. } \displaystyle\frac{1}{2}\left(1 + \frac{q-m_1}{\sqrt{q^2-2m_1q+m_2}}\right),\\
\displaystyle q + \sqrt{q^2-2m_1q+m_2}, & \textrm{w.p. } \displaystyle\frac{1}{2}\left(1 - \frac{q-m_1}{\sqrt{q^2-2m_1q+m_2}}\right),
\end{array}\right.
\end{equation*}
where the support points and the probabilities are dependent on $q$ (this can be viewed as the ``power" of the adversary). In the case, when the order quantity $q$ lies in the range $[0, m_2/2m_1]$, the worst-case demand distribution is two-point, but independent of $q$ and given by $\tilde{d}^*_{m_2/2m_1}$, where:
\begin{equation*} \label{d}
\displaystyle \tilde{d}^*_{m_2/2m_1} = \left\{\begin{array}{llr}
\displaystyle 0, & \textrm{w.p. } \displaystyle 1-\frac{m_1^2}{m_2},\\
\displaystyle \frac{m_2}{m_1}, & \textrm{w.p. } \displaystyle \frac{m_1^2}{m_2}.
\end{array}\right.
\end{equation*}
Combining these results, the worst-case bound is given as:
\begin{equation} \label{e2}
\displaystyle \sup_{F \in \mathcal{F}_{1,2}} \:\mathbb{E}_F[\tilde{d} - q]_+ = \left\{\begin{array}{ll}
\displaystyle \frac{1}{2}\left({\sqrt{q^2-2m_1q+m_2}}-(q-m_1)\right), & \textrm{if } \displaystyle q \geq \frac{m_2}{2m_1},\\
\displaystyle m_1 - \frac{qm_1^2}{m_2}, & \textrm{if } \displaystyle 0\leq q < \frac{m_2}{2m_1}.
\end{array}\right.
\end{equation}
Plugging in the expression (\ref{e2}) into (\ref{c}) and a direct application of calculus provides a closed form solution for the optimal order quantity as follows:
\begin{equation*} \label{e}
q^* = \left\{\begin{array}{ll}
\displaystyle m_1 +\frac{\sqrt{m_2-m_1^2}}{2}\frac{2\eta - 1}{\sqrt{\eta(1-\eta)}}, & \textrm{if } \displaystyle \frac{m_2-m_1^2}{m_2} < \eta < 1,\\
0,& \textrm{if } \displaystyle 0 \leq \eta < \frac{m_2-m_1^2}{m_2},
\end{array}\right.
\end{equation*}
where in the case that $\eta = {(m_2-m_1^2)}/{m_2}$, the optimal order quantity is the given by the set of all values in the interval $[0,{m_2}/{(2m_1)}]$.

While the optimal order quantity is a simple closed form expression, this model is criticized at times for being conservative\footnote{We cite from page 243 in Wang, Glynn and Ye \cite{zizhuo}: ``In the distributionally robust optimization approach, the worst-case distribution for a decision is often unrealistic. Scarf (1958) shows that the worst-case distribution in the newsvendor context is a two-point distribution. This raises the concern that the decision chosen by this approach is guarding under some overly conservative scenarios, while performing poorly in more likely scenarios. Unfortunately, these drawbacks seem to be inherent in the model choice and cannot be remedied easily."}. However it is also known that for certain distributions, the model provides a good approximation. Scarf \cite{scarf58} is his original treatise had observed that for a large range of critical ratios (specifically $\eta$ in the range $[0.05,0.95]$), the optimal order quantity for the two moment model is very close to optimal order quantity for a normal approximation of a Poisson distribution, while for higher critical ratios, the model prescribed higher order quantities. Gallego and Moon \cite{gallego93} in a follow-up set of experiments compared the order quantity from Scarf's model with the optimal order quantity for normally distributed demands and concluded that for a large range of critical ratios, the loss in profit is not significant. \textcolor{black}{On the other hand, Wang, Glynn and Ye \cite{zizhuo} found numerically that the difference in the order quantity from Scarf's model and the true optimal order quantity under the exponential distribution is more significant, for certain choices of the critical ratio (specifically they consider critical ratios around 0.5 in Figure \ref{plotresults1}(c)).}
In Figure \ref{plotresults1}, we provide a comparison of the optimal order quantities for  demand distributions (normal and exponential) and Scarf's model where only the first two moments are assumed to be known. While the figure suggests, the optimal order quantities from Scarf's model is comparatively close to the optimal order quantities for the normal and exponential distributions for moderate critical ratios, it prescribes substantially higher order quantities for high critical ratios. In this paper, we provide an analytical characterization of this numerical insight and generalize it to a larger class of models.
 \begin{figure}[htbp]
\centering
\includegraphics[width=18cm] {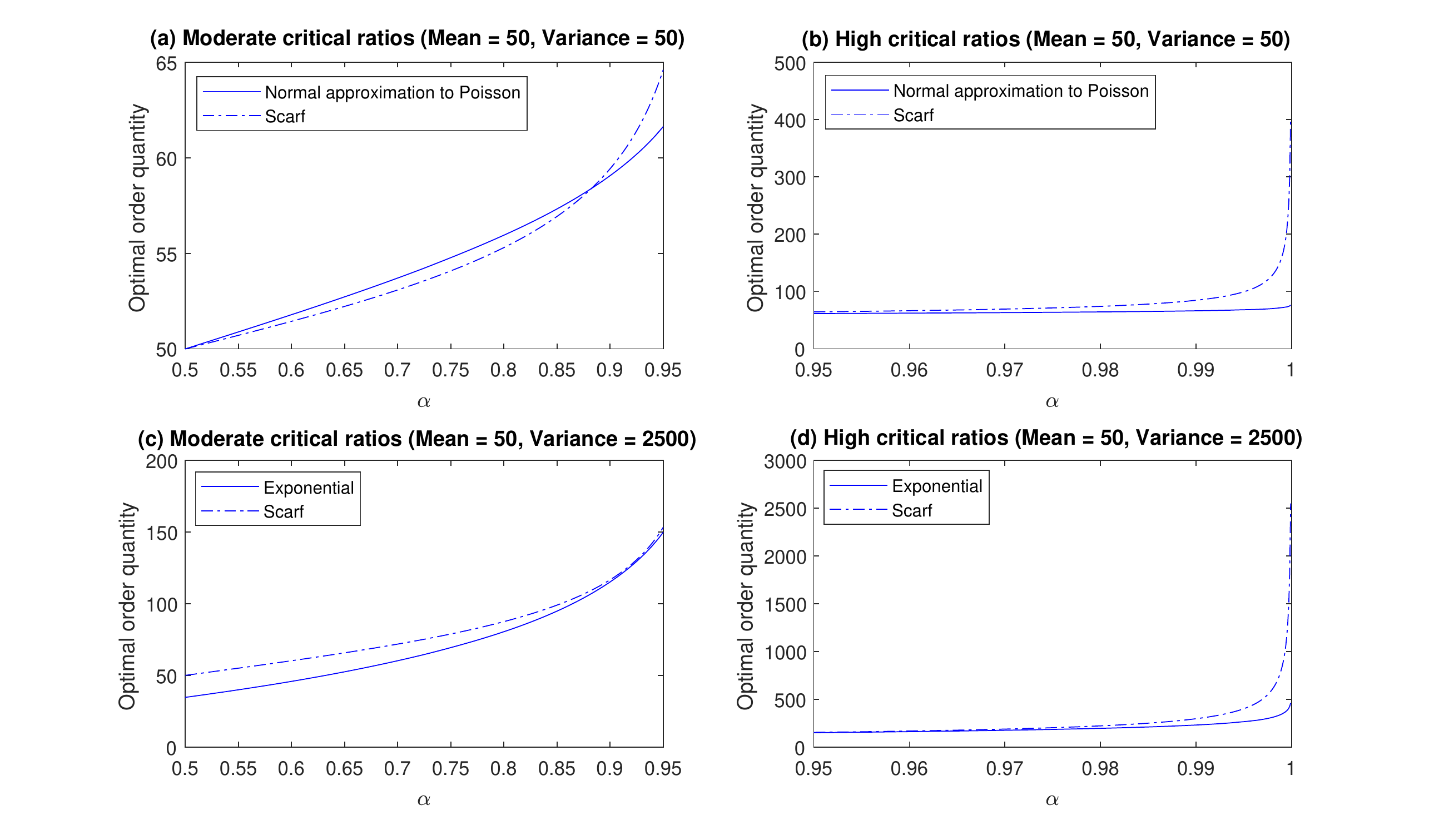}
\caption{The plots at the top compare the optimal order quantities for a normal approximation to a Poisson demand distribution and Scarf's model as in the original paper of Scarf \cite{scarf58} with mean demand $50$ and variance $50$. The plots at the bottom compare the optimal order quantities for an exponential demand distribution and Scarf's model with mean demand $50$ and variance $2500$.}
\label{plotresults1}
\end{figure}

The format of the paper and the main contributions are discussed next:
\texitem{(a)} In Section 2, we provide an overview of the distributionally robust newsvendor problem while analytically characterizing a heavy-tail optimality property that the Scarf's newsvendor model possesses. While this observation has been made in prior research (see M\"{u}ller and Stoyan \cite{muller} and Gallego \cite{gallego98}), the result has not been extended to generalizations of Scarf's model, to the best of our knowledge. We also discuss empirical research that suggests that heavy-tailed distributions is a possible occurrence in real demand datasets.
\texitem{(b)} \textcolor{black}{In Section 3, we propose a generalization of the ambiguity set from the first and the second moment to the first and the $\alpha$th moment for any real $\alpha > 1$. The ambiguity set is simple while providing flexibility in allowing for new sets of distributions to be considered. Specifically, for $1 < \alpha < 2$, the ambiguity set contains distributions that are more heavy-tailed in comparison to Scarf's model (second moments are not necessarily finite), while for $\alpha > 2$, it contains distributions that are more light-tailed than in Scarf's model. However, unlike Scarf's model, there is a technical challenge since the worst-case expected cost does not appear to have a closed form representation in general. In the special case, when the order quantities are below a certain value, we derive a tight closed form expression. Based on this, we show that for any $\alpha$, there exists a threshold for the critical ratio below which it is optimal to order nothing. On the other hand, in the case, when order quantities are above a certain value, we derive new upper and lower bounds for the worst-case expected cost by creating appropriate primal and dual feasible solutions to the moment problem. We show that these upper and lower bounds are asymptotically tight, namely the ratio of the bounds converge to 1 as the order quantity approaches infinity. These bounds help us characterize the tail behavior of the optimal order quantity.}
\texitem{(c)} In Section 4, we provide a characterization of the optimal order quantity in the high service level regime for the distributionally robust newsvendor model by showing that it is optimal for a regularly varying distribution with tail parameter $\alpha$, using techniques developed to model heavy-tailed distributions. This provides an explicit link between the solution of a robust optimization problem which accounts for worst-case behavior and heavy-tails which are used to model extreme events. Particularly, it shows that while the worst-case distribution in a distributionally robust newsvendor problem might be discrete with a few support points, the order quantities remain optimal for high critical ratios, for a regularly varying continuous distribution with an infinite $\alpha$th moment.
\texitem{(d)} \textcolor{black}{In Section 5, we provide numerical examples to illustrate the usefulness of the model for high critical ratios by showing: (1) the value of incorporating moment information beyond the variance in better approximating the optimal order quantities for specific distributions, (2) the difference in the behavior of the ratio of the optimal order quantity and the optimal cost in the robust model in comparison to distributions that are not in the regularly varying class, and, (3) the robustness of the optimal order quantity when a standard distribution is contaminated with a heavy-tailed (regularly varying) distribution, even with a small amount of contamination.  We finally conclude in Section 6 by identifying future research directions.}




\section{Literature Review} \label{two}
In this section, we review some of the key results for the distributionally robust newsvendor problem with a focus on ambiguity sets where demand might take any value in $[0,\infty)$. Our interest in such ambiguity sets stems from an attempt to provide a characterization of the tail behavior which is of particular interest in solving newsvendor problems with high service levels. There is growing evidence in the literature that a stockout for retailer has significant short-term and long-term effects that needs to be minimized (see Anderson, Fitzsimons and Simester \cite{and}). The high service level regime is the natural domain of interest in this case. We also review prior research that provides empirical evidence on heavy-tailed demand distributions.

\subsection{Distributionally Robust Newsvendor Models} \label{related}
Shapiro and Kleywegt \cite{shapiro02} and Shapiro and Ahmed \cite{shapiro} developed a reformulation of the distributionally robust newsvendor as a classical newsvendor problem through the construction of a new probability demand distribution. \textcolor{black}{The key insight to their reformulation is the observation (see Section 3.1 on page 532 in Shapiro and Kleywegt \cite{shapiro02}) that given a set $\mathcal{F}$ with a finite mean, there exists a non-negative random variable $\tilde{d}^*$ with probability distribution ${F}^*$ such that the following equality holds for all values of $q$:
\begin{equation}  \label{f}
\displaystyle \sup_{F \in \mathcal{F}} \:\mathbb{E}_{F}[\tilde{d} - q]_+ = \mathbb{E}_{{F}^*}[\tilde{d}^* - q]_+,~\forall q.
\end{equation}
This is seen by noting that the function $\Pi(q) := \sup_{F \in \mathcal{F}} \:\mathbb{E}_{F}[\tilde{d} - q]_+$ is a non-increasing, convex function that satisfies the following properties: (i) For all $q \leq 0$, $\Pi(q)+q = \sup_{F \in \mathcal{F}} \:\mathbb{E}_{F}[\tilde{d}]$ which is a constant under the assumption that the ambiguity set contains distributions with a finite mean (equal to $m_1$ when the mean is specified in the ambiguity set), and (ii) $\lim_{q \rightarrow \infty}\Pi(q) = 0 $. This implies that there exists a non-negative random variable $\tilde{d}^*$ with a distribution given by $F^*(q) = 1 + \Pi'_{+}(q)$ where $\Pi'_{+}(\cdot)$ is the right derivative of $\Pi(\cdot)$, such that condition (\ref{f}) is satisfied (see Theorem 1.5.10 in M\"{u}ller and Stoyan \cite{muller} for a related statement). It is easy to verify that when the ambiguity set $\mathcal{F}$ consists of random variables with a fixed mean $m_1$, the new random variable $\tilde{d}^*$ also has mean $m_1$, since $\Pi(0) = m_1$.} This corresponds to a random variable that dominates all the random variables $\tilde{d}$ in the set $\mathcal{F}$ in an increasing convex order sense (see M\"{u}ller and Stoyan \cite{muller}, Shaked and Shanthikumar \cite{shaked}). Unlike the extremal distribution on the left hand side of the equation (\ref{f}) which might vary with $q$, the random variable $\tilde{d}^*$ on the right hand side has a distribution $F^*$ which is independent of $q$. This equivalence helps convert the distributionally robust newsvendor problem to the classical newsvendor problem as follows:
\begin{equation}
\begin{array}{lll} \label{c1}
\displaystyle \min_{q \in \Re_{+}}  & \displaystyle  \left((1-\eta)q+\mathbb{E}_{F^*}[\tilde{d}^*-q]_{+} \right),
\end{array}
\end{equation}
where $F^*$ is independent of the critical ratio $\eta$. However, the challenge in applying this technique to solve the distributionally robust newsvendor problem is that $F^*$ in most cases does not have an explicit characterization in terms of the original set of distributions $\mathcal{F}$ and might not even lie in this set. However, the equivalence provides an important insight as it identifies a new distribution $F^*$ for which the optimal order quantity from the distributionally robust newsvendor model in (\ref{c}) is optimal, regardless of the parameters of the problem.

\textcolor{black}{In Scarf's model, it is straightforward to construct the distribution ${F}^*$ as the worst-case bound is known in closed form for each value of $q$. Computing the right-derivative of the worst-case bound in the two regions in (\ref{e2}) and using $F^*(q) = 1 + \Pi^{'}_{1,2+}(q)$, where $\Pi_{1,2}(q) = \sup_{F \in \mathcal{F}_{1,2}} \:\mathbb{E}_{F}[\tilde{d} - q]_+$,} we obtain the characterization of the distribution $F^*$ as follows:
\begin{equation} \label{f1}
\displaystyle F^{*}(w) = \mathbb{P}(\tilde{d}^* \leq w) = \left\{\begin{array}{ll}
\displaystyle  \displaystyle\frac{1}{2} \left(1 + \frac{w-m_1}{{\sqrt{w^2-2m_1w+m_2}}}\right), & \textrm{if } \displaystyle w \geq \frac{m_2}{2m_1},\\
\displaystyle  \displaystyle 1-\frac{m_1^2}{m_2}, & \textrm{if } \displaystyle 0\leq w < \frac{m_2}{2m_1}.
\end{array}\right.
\end{equation}
The distribution in (\ref{f1}) defines a censored student-t random variable with a mixture of discrete and continuous terms as follows:
\begin{equation} \label{g1}
\displaystyle  \tilde{d}^* = \left\{\begin{array}{ll}
\displaystyle \tilde{t}_{\footnotesize{2}}\left(\tiny{m_1,(m_2-m_1^2)/2}\right), & \textrm{if } \displaystyle \tilde{t}_{\footnotesize{2}}\left(\tiny{m_1,(m_2-m_1^2)/2}\right) \geq \frac{m_2}{2m_1},\\
\displaystyle 0, & \textrm{otherwise},
\end{array}\right.
\end{equation}
where $\tilde{t}_{\nu}(\mu,\sigma^2)$ is a three parameter student-t random variable with location parameter $\mu$, scale parameter $\sigma > 0$ and degrees of freedom parameter $\nu > 0$ with a probability density function given by:
\begin{equation}\label{h1}
\begin{array}{ll}
g(w) = \displaystyle\frac{\Gamma(\frac{\nu+1}{2})}{\sqrt{\pi \nu} \sigma \Gamma(\frac{\nu}{2})} \left(1 + \frac{1}{\nu}\left(\frac{w-\mu}{\sigma}\right)^2\right)^{-(\nu+1)/2},  & \forall w\in \Re.
\end{array}
\end{equation}
The distribution of the censored student-t random variable in (\ref{g1}) is defined by a discrete probability mass function:
\begin{equation*} \label{e1}
\displaystyle \mathbb{P}(\tilde{d}^* = 0) =  \displaystyle \frac{m_2-m_1^2}{m_2},
\end{equation*}
and a continuous probability density function given by:
\begin{equation*} \label{e3}
\displaystyle {{f}^*}(w) =  \displaystyle\frac{1}{2}\frac{m_2-m_1^2}{({w^2-2m_1w+m_2})^{3/2}}, ~\forall w \geq \frac{m_2}{2m_1}.
\end{equation*}
A straightforward calculation indicates that for this random variable, the second moment is infinite, that is:
\begin{equation*}\label{i1}
\begin{array}{rllll}
\displaystyle \mathbb{E}_{F^*}[\tilde{d}^{*2}] & = & \infty.
\end{array}
\end{equation*}
This implies that to recreate the optimal order quantity of the distributionally robust newsvendor model under the assumption of a known mean and variance, we need to solve a standard newsvendor problem with a censored student-t distribution with parameter $2$. Thus the demand distribution in the standard newsvendor model has to possess infinite variance which is an heavy-tail property to recreate the distributionally robust newsvendor solution with a finite variance. \textcolor{black}{To the best of our knowledge, this observation has been made for real-valued random variables with known first and second moments in Theorem 1.10.7 on page 57 in M\"{u}ller and Stoyan \cite{muller} with its application to inventory problems discussed in Gallego \cite{gallego98}. Under the assumption that the set of distributions consists of real-valued random variables (not necessarily nonnegative) with fixed first moment $m_1$ and second moment $m_2$, M\"{u}ller and Stoyan \cite{muller} showed that the distribution $F^*$ is given as:
\begin{equation} \label{f11111}
\begin{array}{rlllll}
\displaystyle F^{*}(w) = \mathbb{P}(\tilde{d}^* \leq w) & = & \displaystyle\frac{1}{2} \left(1 + \frac{w-m_1}{{\sqrt{w^2-2m_1w+m_2}}}\right), & \forall w.
\end{array}
\end{equation}
which simply corresponds to the first term in (\ref{f1}). The representation in (\ref{f1}) generalizes this to non-negative demand random variables, which is of interest in the newsvendor setting.}

Since the pioneering work of Scarf \cite{scarf58}, there have been several generalizations of the distributionally robust newsvendor model to new ambiguity sets. While for some of these ambiguity sets (see Ben-Tal and Hochman \cite{bental0,bental}, Natarajan, Uichanco and Sim \cite{Nat17}, \textcolor{black}{Bertsimas, Gupta and Kallus \cite{kallus}, Chen et. al. \cite{chen}}), the problem is solvable in a closed form-manner, in most cases, numerical optimization techniques are needed.
Bertsimas and Popescu \cite{bertsimas02,bertsimas05} and Lasserre \cite{jean02} developed semidefinite optimization techniques to compute the worst-case bound when the set of distributions is defined by a set of fixed moments up to an integer degree $n \in \mathbb{Z}_+$:
\begin{equation*}
\mathcal{{F}}_{1,2,\ldots,n} = \left\{F\in {\mathbb M}(\Re_+): \displaystyle \int_{0}^\infty \mathrm dF(w) = 1, \displaystyle\int_{0}^\infty w^i\mathrm dF(w) = m_i,~i=1,2,\ldots,n\right\}.
\end{equation*}
An application of duality in moment problems implies that the distributionally robust newsvendor problem is solvable as a semidefinite program. While some attempts has been made to solve this problem analytically for $n = 3$ and $n = 4$, the tight worst-case bounds have complicated expressions involving roots of cubic and quartic equations (see Jansen, Haezendonck, and Goovaerts \cite{jansen}, Zuluaga, Pena and Du \cite{zuluaga}). Popescu \cite{pop} generalized these bounds by incorporating additional structural properties such as symmetry and unimodality to the ambiguity sets. Semidefinite optimization and second order conic optimization methods have been developed to find the worst-case bounds for such problems under structural information (see Popescu \cite{pop}, Van Parys, Goulart and Kuhn \cite{van}, Li, Jiang and Mathieu \cite{li}). In general, for these problems, there is an absence of closed form solutions and hence finding an explicit representation of the distribution $F^{*}$ does not appear to be straightforward.

Lam and Mottet \cite{lam} have recently proposed an ambiguity set, where information on the tail probability of the random variable beyond a given threshold, the density function at the threshold and the left derivative of the density function at the threshold is known with an additional assumption that the tail density function is convex. \textcolor{black}{Under this ambiguity set, they showed that the worst-case distribution is either extremely light-tailed or extremely heavy-tailed depending on a characterization of the limit of the objective function at the tail and proposed the use of low dimensional nonlinear optimization methods to compute the corresponding bound. In the newsvendor problem, their result implies that the worst-case density function is in fact a piecewise linear function beyond the threshold, with at most two linear pieces and is a extremely light-tailed distribution (see Theorem 4 and Section 7 in Lam and Mottet \cite{lam}).} In contrast to their approach which models the tail behavior in the ambiguity set, we focus on ambiguity sets with moment information and characterize the tail behavior implied by the model. Ben-Tal et al. \cite{bentalms} studied newsvendor problems with $\phi$-divergence based ambiguity sets around a reference discrete distribution and proposed a convex optimization formulation to solve the distributionally robust optimization problem. Blanchet and Murthy \cite{murthy} build on this model to show that under the assumption that the reference measure is a distribution such as exponential, Weibull or Pareto distribution, the use of an ambiguity set with a Kullback-Leibler distance measure contains distributions where the tail probabilities decay at a very slow rate for which the worst-case expected costs might be infinite. To overcome this pessimism, they proposed alternative ambiguity sets using the Renyi divergence measure for which the worst-case tails are heavier than the reference measure, but not as heavy as the Kullback-Leibler divergence measure. In contrast, we focus on the moment ambiguity set in this paper. \textcolor{black}{A related recent stream of literature has focused on solving distributionally robust optimization problems, including the newsvendor model, using ambiguity sets defined around an empirical distribution with the Wasserstein distance (see Esfahani and Kuhn \cite{esf}, Gao and Kleywegt \cite{gao} and Blanchet and Murthy \cite{murthy1}) with convex optimization methods. Under a light-tailed assumption on the underlying distribution from which the empirical distribution is generated, it is possible to obtain finite sample guarantees with this ambiguity set. However, these results do not extend in a straightforward manner to heavy-tailed distributions. Furthermore, under the assumption that the demand can take any value in $[0,\infty)$, the distributionally robust order quantity with the Wasserstein distance is the same as the solution to the empirical distribution (see Remark 6.7 in Esfahani and Kuhn \cite{esf})} We next review empirical research that provides evidence on the existence of heavy-tailed distributions in demand data.

\subsection{Empirical Evidence of Heavy Tailed Demand} \label{emp}
There has been growing evidence in the recent years that heavy-tailed demand distributions can occur in practice and has to be better accounted for in operational settings. Clauset, Shalizi and Newman \cite{clauset} in their well-cited study on the presence of power law distributions in real world datasets, developed a set of statistical tests to help validate if the data follows a power law. They studied twenty four different datasets across a broad range of disciplines in physics, earth sciences, biology and engineering where prior research in these domains had conjectured that the data followed a power law. Of these in seventeen of the datasets, the statistical tests provided evidence that the power law hypothesis was a reasonable one and could not be firmly ruled out while in the remaining seven datasets, the p-values were too small and with reasonable confidence, the power law could be ruled out. Among these datasets, two of them which are particularly relevant to demand models are: a) the number of calls received by customers of AT\&T long distance telephone in the United States during a single day and b) the number of copies of bestselling books sold in the United States during the period 1895 to 1965. In both these datasets, the authors found strong evidence that the power law tail is a reasonable model in comparison to the exponential and stretched exponential distributions but at the same time it was not possible to rule out other heavy-tail distributions such as the lognormal distribution as a possible fit. In another study, Gaffeo, Scorcu and Vici \cite{gaffeo} analyzed the demand of books in Italy and found that for the three categories - local novels, foreign novels and non-fiction books, a power law distribution where the exponent is typically lesser than $2$ is a good fit to the right tail of the demand distribution. Bimpkis and Markakis \cite{bimkis} used the ratings of movies on Netflix as an approximation to the demand of a movie and estimated a power law distribution with an exponent of around 1.04 for the number of movies per number of distinct ratings. Using data from a North American retailer over a one year period with 626 products, their statistical tests showed that the exponential and normal distributions were a poor fit to the data while the power law provided a reasonable approximation to the dataset. Building on this observation, they showed that for a class of heavy-tailed stable demand distributions, the benefits from pooling in inventory can be much lower than that predicted for normally distributed demands. Natarajan, Sim and Uichanco \cite{Nat17} used data from an European automotive manufacturer with 36 spare part SKUs over a one year period. In fitting demand distributions to the data over 17 different families, they found that the best-fit was often obtained by heavy-tailed distributions such as Pareto, extreme value or t-distributions. Chevalier and Goolsbee \cite{chevalier} used publicly available data on sales ranks of books from the online book retailer Amazon.com to obtain an estimate on the sales quantity of the books. In their numerical experiments, they identified that the Pareto distribution with a parameter of 1.2 was a reasonable approximation to the demand data. The Internet has particularly fueled the phenomenon of the long tail where niche products gives rise to a large share of the total sales for online retailers, popularly referred to as the long-tail phenomenon (see Anderson \cite{anderson}, Brynjolfsson, Hu and Simester \cite{bryn}). Empirical evidence in this literature seems to suggest that when Pareto distributions are used to model the demand, the exponent is strictly greater than 1 and possesses finite mean but might not necessarily possess finite variance. The ambiguity set we consider in the next section is inspired from this empirical evidence.

\section{Model with the First and $\alpha$th Moment}\label{sec2}

Consider an ambiguity set defined as follows:
\begin{equation} \label{nmom}
\mathcal{F}_{1,\alpha} = \left\{F \in {\mathbb M}(\Re_+): \displaystyle\int_{0}^\infty \mathrm dF(w) = 1,~ \displaystyle\int_{0}^\infty w\: \mathrm d F(w) = m_1,~ \displaystyle\int_{0}^\infty w^{\alpha}\: \mathrm d F(w) = m_{\alpha}\right\},
\end{equation}
where $\alpha >1$ is an arbitrary real number and $m_1$ and $m_{\alpha}$ are the first and the $\alpha$th moment respectively, satisfying $m_{\alpha} > m_1^{\alpha} > 0$. Note that for $m_{\alpha} = m_1^{\alpha}$, the only feasible distribution in the ambiguity set is the demand $m_1$ that occurs with probability $1$, which is a trivial case to deal with. We discuss a few features of this ambiguity set next:
\texitem{(a)} The ambiguity set $\mathcal{F}_{1,\alpha}$ can be used with any real value $\alpha > 1$, not necessarily just an integer. Clearly, when $\alpha=2$, this corresponds to the original model of Scarf \cite{scarf58}. This provides a natural generalization of the set $\mathcal{F}_{1,2}$, but allows for the possibility of the ambiguity set to specify more light tailed ($\alpha > 2$) or more heavy tailed distributional information ($\alpha < 2$) than Scarf's model allows. In conjunction with the empirical evidence discussed in Section \ref{emp}, assuming the knowledge of a finite mean also seems reasonable in most applications involving real demand data.
\texitem{(b)} The ambiguity set preserves the simplicity of Scarf's \cite{scarf58} moment ambiguity set as it is parameterized by the choice of only three parameters - $m_1$, $m_{\alpha}$ and $\alpha$. \textcolor{black}{The choice of $\alpha$ can be estimated from sample data using nonparametric hypothesis tests such as the one proposed in Fedotenkov \cite{fedotenkov} where the null hypothesis is that the $\alpha$th moment exists while the alternate hypothesis is that $\alpha$th moment does not exist. This nonparametric bootstrap test builds on the observation of Derman and Robbins \cite{derman} that when a certain moment is infinite the moments of the sample from the distribution grows faster than the moments of subsamples of a smaller size, under certain regularity assumptions. The availability of such statistical tests aids in calibrating $\alpha$ from the data.}

\textcolor{black}{Under this ambiguity set, we are interested in solving:
\begin{equation}
\begin{array}{lll} \label{cnew}
\displaystyle \min_{q \in \Re_{+}} & \displaystyle  \left((1-\eta)q + \sup_{F \in \mathcal{F}_{1,\alpha}}\mathbb{E}_{F}[\tilde{d}-q]_+\right).
\end{array}
\end{equation}
The flexibility of allowing for any $\alpha > 1$ however leads to challenges in solving the inner moment problem in closed form, unlike the $\alpha = 2$ case.
To see why, we consider the primal formulation for the worst-case expected value $\sup_{F \in \mathcal{F}_{1,\alpha}} \:\mathbb{E}_F[\tilde{d} - q]_+$ which is given as:
\begin{equation} \label{primalf0}
\begin{array}{rlll}
\sup & \displaystyle \int_{0}^\infty [w - q]_+d F(w) &\\
\textrm{s.t.} & \displaystyle\int_{0}^\infty \mathrm d F(w) = 1,  & \\
& \displaystyle\int_{0}^\infty w\mathrm d F(w) = m_1,  & \\
& \displaystyle\int_{0}^\infty w^{\alpha} \mathrm d F(w) = m_{\alpha}, \\
& \displaystyle F\in {\mathbb M}(\Re_+),
\end{array}
\end{equation}
and the corresponding dual formulation given as:
\begin{equation} \label{dualf}
\begin{array}{rlll}
\inf & y_0 + y_1 m_1 + y_{\alpha} m_{\alpha} &\\
\textrm{s.t.} & y_0 + y_1 w + y_{\alpha} w^{\alpha} \geq 0,  & \forall w \geq 0,\\
& y_0 + y_1 w + y_{\alpha} w^{\alpha} \geq w - q, & \forall w \geq 0,
\end{array}
\end{equation}
where $y_0$ is the dual variable for the constraint that the total probability is equal to 1 and $y_1$ and $y_{\alpha}$ are the dual variables for the first and the $\alpha$th moment constraints respectively. The standard attempt to find the closed form solution to such a problem is to try and find the explicit roots of the polynomial equations in the dual formulation which are of the form $aw^{\alpha}+bw+c=0$, if possible. This is easy to do, for example, when $\alpha = 2$ (see Scarf \cite{scarf58}, Bertsimas and Popescu \cite{bertsimas02}), using the solution to quadratic equations. Unfortunately, the Abel-Ruffini impossibility theorem states that there is no solution in terms of radicals (involving only taking roots and the four basic arithmetic operations) for polynomial equations of degree five or higher with arbitrary coefficients. Hence, the state of art approaches to solve problems with higher order moments is through semidefinite optimization (see Bertsimas and Popescu \cite{bertsimas05}, Lasserre \cite{jean02}). Furthermore, even with structured polynomials in the dual, as in our case, it is not possible to find closed form solutions in terms of radicals.  For example, the Bring-Jerrard quintic equations of the form $aw^{5}+bw+c=0$ (where $\alpha$ = 5) does not have a solution in terms of radicals for general values of $a, b$ and $c$ (an example of such a quintic equation is $w^5-w+1=0$ which is discussed on page 121 in Lang \cite{lang}). This makes the solution of the dual formulation in closed form for general parameter values very unlikely. 
In the next section, we consider a very special case where for a range of $q$, it is possible to obtain a closed form expression. This helps us characterize the optimal order quantity in the low service level regime by showing that there exists a threshold below which for all critical ratios, it is optimal to order nothing. We then consider the high service level regime and focus on finding lower and upper bounds on the worst-case expected value that is valid beyond a certain value of the order quantity $q$. Our approach is based on constructing approximately optimal primal-dual solutions that attains the bounds in this regime. Building on this, we provide in Section \ref{sec4}, a characterization of the tail behavior of the distribution of $F^*$ and the optimal order quantity for high critical ratios.
}

\textcolor{black}{
\subsection{Small Values of $q$} \label{lb1}
We consider a special case where the moment problem (\ref{primalf0}) can be solved in closed form for any real number $\alpha > 1$. Building on this, we provide a characterization of the optimal solution to the distributionally robust newsvendor problem (\ref{cnew}) for small values of $\eta$.
\begin{prop}\label{prop1a}
Given an ambiguity set $\mathcal{F}_{1,\alpha}$, the worst-case expected value is given as:
\begin{equation} \label{new1}
\begin{array}{rllll}
\displaystyle \sup_{F \in \mathcal{F}_{1,\alpha}} \:\mathbb{E}_F[\tilde{d} - q]_+ = m_1 - q \left(\frac{m_1^{\alpha}}{m_{\alpha}}\right)^{1/(\alpha-1)}, & \textrm{if } \displaystyle 0\leq q \leq \left(\frac{\alpha-1}{\alpha}\right) \left(\frac{m_{\alpha}}{m_1}\right)^{1/(\alpha-1)}.
\end{array}
\end{equation}
The worst-case demand distribution in this case is given as:
\begin{equation} \label{new2}
\displaystyle \tilde{d}^*= \left\{\begin{array}{llr}
\displaystyle 0, & \textrm{w.p. } \displaystyle 1-\left(\frac{m_1^{\alpha}}{m_{\alpha}}\right)^{1/(\alpha-1)},\\
\displaystyle \left(\frac{m_{\alpha}}{m_1}\right)^{1/(\alpha-1)}, & \textrm{w.p. } \displaystyle \left(\frac{m_1^{\alpha}}{m_{\alpha}}\right)^{1/(\alpha-1)}.
\end{array}\right.
\end{equation}
\end{prop}
\begin{proof}
We show tightness of the worst-case expected value by constructing a primal and dual feasible solution that attains it for the specified range of $q$. Observe that the demand distribution in (\ref{new2}) is a feasible distribution in the set $\mathcal{F}_{1,\alpha}$, since $\mathbb{E}[\tilde{d}^{*}] = m_1$ and $\mathbb{E}[\tilde{d}^{*\alpha}] = m_{\alpha}$ and attains the expected value in (\ref{new1}). We next construct a dual solution as follows:
\begin{equation}\label{dualnew}
\begin{array}{rcl}
y_0 & = & \displaystyle 0,\\
y_1 & = & \displaystyle 1- q\left(\frac{\alpha}{\alpha-1}\right)\left(\frac{m_1}{m_{\alpha}}\right)^{1/(\alpha-1)},\\
y_{\alpha} & = & \displaystyle \frac{q}{\alpha-1}\left(\frac{m_1}{m_{\alpha}}\right)^{\alpha/(\alpha-1)}.
\end{array}
\end{equation}
We first validate that this solution is dual feasible for $q \in [0,((\alpha-1)/\alpha)(m_{\alpha}/m_1)^{1/(\alpha-1)}]$. To see this, observe that in the specified range for $q$, we have $y_1 \geq 0$. Furthermore since $y_0 = 0$ and $y_{\alpha} \geq 0$, the first dual feasibility constraint in (\ref{dualf}) is satisfied in a straightforward manner. The second dual feasibility constraint can be expressed as:
\begin{equation}\label{newbd}
 \displaystyle\min_{w \geq 0} \left(y_0 + q + (y_1 - 1) w + y_{\alpha} w^{\alpha} \right) \geq 0.
 \end{equation}
 This constraint is satisfied at equality for $q = 0$, since $y_1 = 1$ and $y_{\alpha} = 0$ in this case. Next, we focus on the case with $q > 0$. Since $y_0 = 0$, $y_{\alpha} > 0$ and $y_1 \in [0,1)$, the minimum value in (\ref{newbd}) is attained at:
\begin{equation*}
\begin{array}{rlll}
w^* & = & \displaystyle \left(\frac{1-y_1}{\alpha y_{\alpha}}\right)^{1/(\alpha-1)}, \\
& = & \displaystyle \left(\frac{m_{\alpha}}{m_1}\right)^{1/(\alpha-1)}.
\end{array}
\end{equation*}
Substituting in the given choice of the dual variables and $w^*$, the left hand side of (\ref{newbd}) reduces to:
\begin{equation*}
\begin{array}{rlll}
\displaystyle y_0 + q + (y_1 - 1) w^{*} + y_{\alpha} w^{*\alpha} & = & \displaystyle q -  q\left(\frac{\alpha}{\alpha-1}\right)+ \frac{q}{\alpha-1}, \\
& = & \displaystyle 0,
\end{array}
\end{equation*}
implying the feasibility of the second dual constraint. Finally, we verify that this dual feasible solution is optimal, since the objective value of the dual feasible solution is given as:
\begin{equation*}
\begin{array}{rlll}
\displaystyle y_0 + y_1 m_1 + y_{\alpha}m_{\alpha} & = & \displaystyle  m_1 - q \left(\frac{m_1^{\alpha}}{m_{\alpha}}\right)^{1/(\alpha-1)},
\end{array}
\end{equation*}
which is equal to the objective value of the primal solution.
\end{proof}
When $\alpha = 2$, Proposition \ref{prop1a} reduces precisely to the second term in the worst-case expected value in (\ref{e2}) as developed by Scarf \cite{scarf58}. Proposition \ref{prop1a} indicates that for any $\alpha > 1$, there is always a range of $q$ around 0, where the worst-case value is a linearly decreasing function of $q$. Building on this closed form expression, we can identify a characterization of the optimal order quantity for small values of the critical ratio $\eta$ where it optimal to order zero units as follows.
\begin{prop}\label{prop1-smallq}
Define the threshold value of the critical ratio, ${\eta}_0 = 1-\left({m_1^{\alpha}}/{m_{\alpha}}\right)^{1/(\alpha-1)} \in (0,1)$.
The optimal order quantity to the distributionally robust newsvendor problem in (\ref{cnew}) satisfies the following properties:
 \texitem{(a)} For any critical ratio in the range $\eta \in [0,{\eta}_0)$, the optimal order quantity is given as:
\begin{equation*} \label{e}
\begin{array}{rllll}
\displaystyle q^*  =  0. &
\end{array}
\end{equation*}
 \texitem{(b)} For any critical ratio in the range $\eta \in ({\eta}_0,1)$, the optimal order quantity is strictly positive and satisfies the condition:
 \begin{equation*} \label{e}
\begin{array}{rllll}
\displaystyle q^* \geq q_0 := \left(\frac{\alpha-1}{\alpha}\right) \left(\frac{m_{\alpha}}{m_1}\right)^{1/(\alpha-1)}. &
\end{array}
\end{equation*}
 \texitem{(c)} When $\eta = {\eta}_0$, the set of optimal order quantities contains all the values in the range $[0,q_0]$.
\end{prop}
\begin{proof}
The optimal solution to the distributionally robust newsvendor problem in (\ref{cnew}) is given as:
\begin{equation}
\begin{array}{lll} \label{newa}
\displaystyle \min\left\{\min_{0 \leq q \leq q_0} (1-\eta)q + \Pi_{1,\alpha}(q), \ \min_{q > q_0} (1-\eta)q + \Pi_{1,\alpha}(q)\right\},
\end{array}
\end{equation}
where $\Pi_{1,\alpha}(q) = \sup_{F \in \mathcal{F}_{1,\alpha}} \:\mathbb{E}_{F}[\tilde{d} - q]_+$. Using Proposition \ref{prop1a}, the first term in this expression reduces to:
\begin{equation}
\begin{array}{rlll} \label{newaaa}
\displaystyle \min_{0 \leq q \leq q_0} (1-\eta)q + \Pi_{1,\alpha}(q) & = & \displaystyle \min_{0 \leq q \leq q_0} m_1 - q \left(\eta-1+\left(\frac{m_1^{\alpha}}{m_{\alpha}}\right)^{1/(\alpha-1)}\right), \\
& = &  \displaystyle \min_{0 \leq q \leq q_0} m_1 - q \left(\eta-\eta_0\right).
\end{array}
\end{equation}
The minimizer in (\ref{newaaa}) is $q^* = 0$ when $\eta < \eta_0$ and the corresponding objective value is $m_1$. Since the objective function in (\ref{cnew}) is convex in $q$ and increasing in the range $[0,q_0]$, the global minimum is also attained at $q^* = 0$ (see Case (a) in Figure \ref{m31aaaa}). Similarly, the minimizer in (\ref{newaaa}) is $q_0$ when $\eta > \eta_0$, since the function is strictly decreasing in this range. Moreover, since the objective function is convex, the global minimum in this case is attained at some $q^* \geq q_0 > 0$ (see Case (b) in Figure \ref{m31aaaa}). Finally, when $\eta = \eta_0$, the minimum value in (\ref{newaaa}) is attained for all $q \in [0,q_0]$ (see Case (c) in Figure \ref{m31aaaa}). Furthermore, since the objective function is convex, the global minimum must also be attained at these values.
\end{proof}
\begin{figure}[htbp]
\centering
\includegraphics[width=18cm] {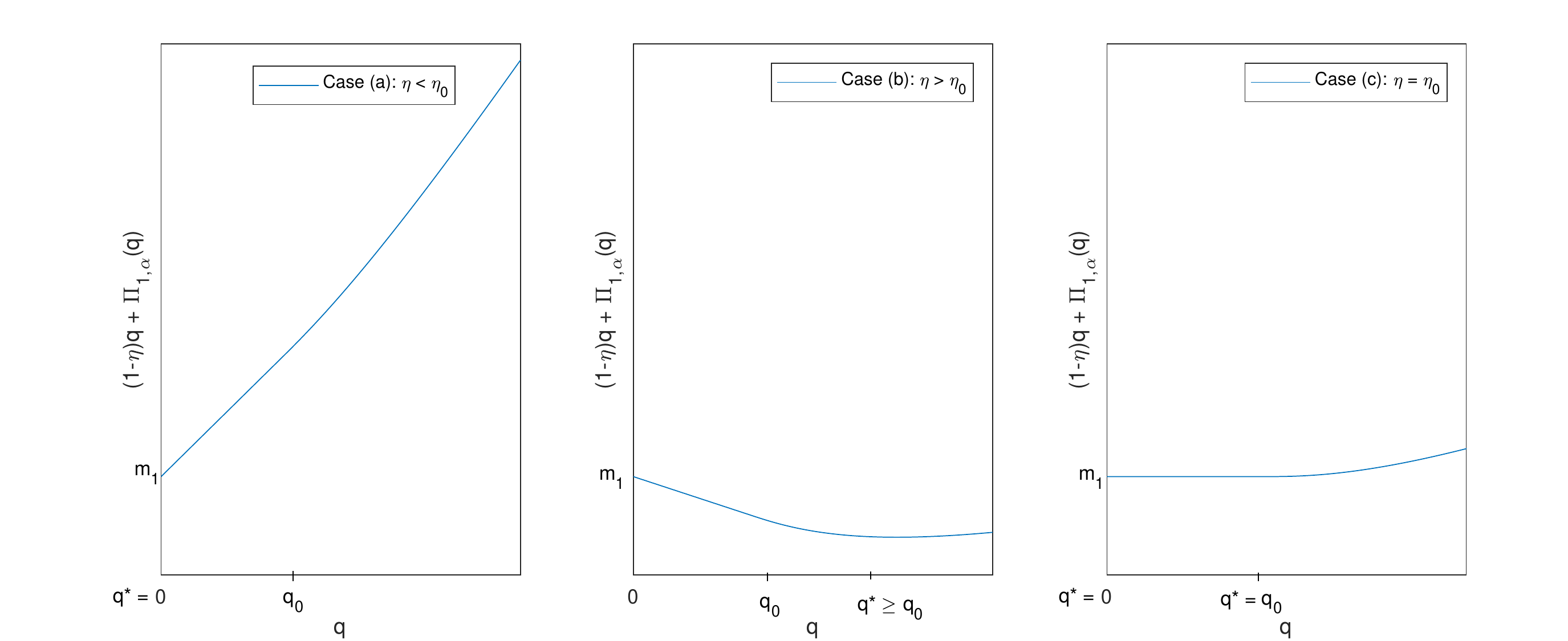}
\caption{Case (a) provides the optimal solution $q^* = 0$ for $\eta < \eta_0$, case (b) provides the optimal solution $q^* \geq q_0 > 0$ for $\eta > \eta_0$ case while case (c) provides the set of optimal solutions for $\eta = \eta_0$.}
\label{m31aaaa}
\end{figure}
Proposition \ref{prop1-smallq} generalizes the result from Scarf's model to identify the range of critical ratios, under which it is optimal to order zero units for any real $\alpha > 1$. This shows that for any finite $\alpha$, regardless of how high the order of the specified moment in the ambiguity set is, there is always a certain range of small critical ratios where it is optimal to order zero for the moment based ambiguity set. The corresponding $F^*$ distribution thus has a finite probability mass at a demand of zero. Proposition \ref{prop1-smallq} also illustrates that when critical ratios are larger than the threshold value, the optimal order quantity is strictly positive. In the next two sections, we identify lower and upper bounds on the worst-case value that help characterize $q^*$ for large values of the critical ratio.}
\subsection{\textcolor{black}{Lower Bound for Large Values of $q$}} \label{lb}
To develop the lower bound, we first consider a related ambiguity set that was studied by Grundy \cite{grundy91} with a fixed $\alpha$th moment only:
\begin{equation}
\mathcal{{F}}_{\alpha} = \left\{F \in {\mathbb M}(\Re_+): \displaystyle\int_{0}^\infty \mathrm dF(w) = 1,~ \displaystyle\int_{0}^\infty w^{\alpha}\:\mathrm dF(w) = m_{\alpha}\right\},
\end{equation}
where $\mathcal{{F}}_{1,\alpha} \subseteq \mathcal{{F}}_{\alpha}$. While Grundy \cite{grundy91} evaluated the worst-case value for this ambiguity set in an option pricing context, the model remains largely unexplored in the newsvendor context. Grundy \cite{grundy91} characterized the unique two point distribution that attains the bound in $\sup_{F \in \mathcal{F}_{\alpha}} \:\mathbb{E}_F[\tilde{d} - q]_+$. Given a value $q > (\alpha-1)m_{\alpha}^{1/\alpha}/{\alpha}$, the worst-case demand distribution was characterized as follows:
\begin{equation} \label{d2}
\tilde{d}^*_q = \left\{\begin{array}{llr}
\displaystyle \frac{q\alpha}{\alpha-1}, & \textrm{w.p. } \displaystyle\frac{{(\alpha-1)}^{\alpha}m_{\alpha}}{\alpha^{\alpha}q^{\alpha}},\\
\displaystyle 0, & \textrm{otherwise},
\end{array}\right.
\end{equation}
while for $0 \leq q \leq (\alpha-1)m_{\alpha}^{1/\alpha}/{\alpha}$, the worst-case demand distribution is degenerate with the mass at the point $m_{\alpha}^{1/\alpha}$. The corresponding worst-case expected value is given as:
\begin{equation*}
\displaystyle\sup_{F \in \mathcal{F}_{\alpha}} \:\mathbb{E}_F[\tilde{d} - q]_+ = \left\{\begin{array}{lll}
\displaystyle\frac{m_{\alpha}}{\alpha} \left(\frac{\alpha-1}{{\alpha}q}\right)^{\alpha-1}, & \textrm{if }q > \displaystyle\frac{\alpha-1}{\alpha} m_{\alpha}^{1/\alpha},\\
\displaystyle m_{\alpha}^{1/\alpha} - q, & \textrm{if }\displaystyle 0 \leq q \leq \displaystyle\frac{\alpha-1}{\alpha} m_{\alpha}^{1/\alpha}.
\end{array}\right.
\end{equation*}
The worst-case distribution in this ambiguity set depends on $q$, as in Scarf's model. In the next proposition, we derive a lower bound on the worst-case expected value by modifying the two point distribution in (\ref{d2}) to a three point distribution to make it feasible for the ambiguity set $\mathcal{F}_{1,\alpha}$.
This brings us to our first main result that provides a lower bound on the worst-case expected value for large values of $q$.

\begin{prop}\label{prop1}
\textcolor{black}{Given an ambiguity set $\mathcal{F}_{1,\alpha}$, the following lower bound is valid:
\begin{equation}\label{eq:prop1}
\begin{array}{rlllll}
\displaystyle\sup_{{F} \in \mathcal{F}_{1,\alpha}}\  \mathbb{E}_{{F}}[\tilde{d}-q]_+ & \geq  &\displaystyle\frac{(m_{\alpha} - m_1^{\alpha})}{\alpha^{\alpha} q^{\alpha-1}} (\alpha-1)^{\alpha-1}, & \forall q > \underline{q}(m_1,m_{\alpha},\alpha),
 \end{array}
 \end{equation}
 for all values of $q$ greater than $\underline{q}(m_1,m_{\alpha},\alpha)$ where: \begin{equation}\label{eq:qunder}
\begin{array}{rlllll}
\underline{q}(m_1,m_{\alpha},\alpha) & := & \left( \left(\displaystyle\frac{m_{\alpha} - m_1^{\alpha}}{m_1}\right) \left(\displaystyle\frac{\alpha-1}{\alpha}\right)^{\alpha-1} + m_1^{\alpha-1}\right)^{1/(\alpha-1)}.
  \end{array}
 \end{equation}
 }
\end{prop}
\begin{proof}
We derive the lower bound through the construction of a three point feasible distribution. The proof is developed in two steps. In Step 1, we provide a three point distribution by a modification of the two point worst-case distribution in (\ref{d2}) such that the moment constraints are met while in Step 2, we show that this defines a valid probability distribution for large values of $q$. Evaluating the objective value for this distribution provides the desired lower bound in (\ref{eq:prop1}).\\
\noindent \textit{Step 1:} Consider a three point random variable $\tilde{d}$ with a distribution defined as follows:
\begin{equation} \label{3point}
\tilde{d} = \left\{\begin{array}{ll}
\displaystyle\frac{q\alpha}{\alpha-1}, &  \textrm{w.p. } \displaystyle\frac{(m_{\alpha} - m_1^{\alpha})}{\alpha^{\alpha}q^{\alpha}} {(\alpha-1)}^{\alpha},\\
w, & \textrm{w.p. } p,\\
0, &\displaystyle \textrm{w.p. } 1 - p - \frac{(m_{\alpha} - m_1^{\alpha})}{\alpha^{\alpha}q^{\alpha}} {(\alpha-1)}^{\alpha},
\end{array}\right.
\end{equation}
where we choose particular values of $w$ and $p$ as discussed next to ensure feasibility. 
Our choice of these values is such that one obtains a strictly positive value of $w$ that is less than $q$ and a probability $p$ such that the first and $\alpha$th moment constraints are met. To do so, we start by ensuring that the $\alpha$th moment constraint for this distribution is met as follows:
\begin{equation*}
\begin{array}{rcl}
m_{\alpha}  &=  &\displaystyle \mathbb{E}[\tilde{d}^{\alpha}], \\
& = &\displaystyle \left(\frac{q\alpha}{\alpha-1}\right)^{\alpha} \displaystyle\frac{(m_{\alpha} - m_1^{\alpha})}{\alpha^{\alpha}q^{\alpha}} {(\alpha-1)}^{\alpha} +  w^{\alpha} p,\\
& = &\displaystyle  m_{\alpha} - m_1^{\alpha} + w^{\alpha} p.
\end{array}
\end{equation*}
This gives rise to a condition that $w$ and $p$ must satisfy:
\begin{equation} \label{eq1}
\begin{array}{rll}
w^{\alpha}p & = & m_1^{\alpha}.
\end{array}
\end{equation}
We next ensure the first moment constraint for the distribution is met as follows:
\begin{equation*}
\begin{array}{rcl}
m_1  &=  &\displaystyle \mathbb{E}[\tilde{d}], \\
& = &\displaystyle \left(\frac{q\alpha}{\alpha-1}\right) \displaystyle\frac{(m_{\alpha} - m_1^{\alpha})}{\alpha^{\alpha}q^{\alpha}} {(\alpha-1)}^{\alpha} +  w p,\\
& = &\displaystyle  \displaystyle\frac{(m_{\alpha} - m_1^{\alpha})}{\alpha^{\alpha-1}q^{\alpha-1}} (\alpha-1)^{\alpha-1} + w p.
\end{array}
\end{equation*}
This gives rise to a second condition that $w$ and $p$ must satisfy:
\begin{equation} \label{eq2}
\begin{array}{rll}
wp & = & \displaystyle m_1 - \frac{(m_{\alpha} - m_1^{\alpha})}{\alpha^{\alpha-1}q^{\alpha-1}} (\alpha-1)^{\alpha-1}.
\end{array}
\end{equation}
Solving the two simultaneous equations (\ref{eq1}) and (\ref{eq2}) gives:
\begin{equation} \label{eq3}
\begin{array}{rll}
w & = & \displaystyle\frac{m_1^{\alpha/(\alpha-1)}}{\left(m_1 - \frac{(m_{\alpha} - m_1^{\alpha})}{\alpha^{\alpha-1}q^{\alpha-1}} (\alpha-1)^{\alpha-1}\right)^{1/(\alpha-1)}},
\end{array}
\end{equation}
\begin{equation} \label{eq3a}
\begin{array}{rll}
 p& = &  \displaystyle\frac{\left(m_1 - \frac{(m_{\alpha} - m_1^{\alpha})}{\alpha^{\alpha-1}q^{\alpha-1}} (\alpha-1)^{\alpha-1}\right)^{\alpha/(\alpha-1)}}{m_1^{\alpha/(\alpha-1)}}.
\end{array}
\end{equation}

\noindent \textcolor{black}{\textit{Step 2:}}
\textcolor{black}{We note that for large values of $q$, the demand realization $w$ in (\ref{eq3}) is smaller than $q$ as it is given by a strictly decreasing function of $q$. This condition is satisfied when:
\begin{equation} \label{l1}
 \begin{array}{rlll}
 q& > & \underline{q}(m_1,m_{\alpha},\alpha) := \displaystyle \left( \left(\displaystyle\frac{m_{\alpha} - m_1^{\alpha}}{m_1}\right) \left(\displaystyle\frac{\alpha-1}{\alpha}\right)^{\alpha-1} + m_1^{\alpha-1}\right)^{1/(\alpha-1)}.
 \end{array}
\end{equation}
Under this condition, since we have only one support point that is strictly above $q$, the expected value of the objective function is given as:
\begin{equation*}
\begin{array}{rlll}
\mathbb{E}[\tilde{d} - q]_+ &= &\displaystyle \left(\frac{q\alpha}{\alpha-1}-q\right) \frac{m_{\alpha} - m_1^{\alpha}}{\alpha^{\alpha}q^{\alpha}}{(\alpha-1)}^{\alpha},\\
& =& \displaystyle\frac{(m_{\alpha} - m_1^{\alpha})(\alpha-1)^{\alpha-1}}
{\alpha^{\alpha} q^{\alpha-1}},
\end{array}
\end{equation*}
which corresponds to the lower bound on the expected value. To complete the proof, we need to ensure that (\ref{3point}) corresponds to a valid probability measure for the chosen $w$ and $p$ for all $q > \underline{q}(m_1,m_{\alpha},\alpha)$ by checking the following four conditions:
\texitem{(a)} To verify that $p  > 0$, we observe that the value of $p$ in (\ref{eq3a}) is strictly positive when:
\begin{equation} \label{l2}
\begin{array}{rll}
 q & > &\underline{q}_1 := \left(\displaystyle\frac{m_{\alpha} - m_1^{\alpha}}{m_1}\right)^{1/(\alpha-1)} \left(\displaystyle\frac{\alpha-1}{\alpha}\right).
\end{array}
\end{equation}
Condition (\ref{l2}) is implied by (\ref{l1}) since:
\begin{equation*}
\begin{array}{rll}
 \underline{q}(m_1,m_{\alpha},\alpha) & = & \displaystyle \left(\underline{q}_1^{\alpha-1} + m_1^{\alpha-1} \right)^{1/(\alpha-1)},\\
 & > &\underline{q}_1,
\end{array}
\end{equation*}
where $m_1 > 0$.
\texitem{(b)} It is easy to verify that $p  <  1$ for $q > \underline{q}(m_1,m_{\alpha},\alpha)$, since $m_{\alpha} > m_1^{\alpha}$. 
\texitem{(c)} We next verify that the probability of the atom $q\alpha/(\alpha-1)$ in (\ref{3point}) is strictly less than $1$. Observe that this condition is satisfied when:
\begin{equation} \label{l3}
\begin{array}{rll}
 q & > &\underline{q}_2 :=\left(\displaystyle{m_{\alpha} - m_1^{\alpha}}\right)^{1/\alpha} \left(\displaystyle\frac{\alpha-1}{\alpha}\right),
\end{array}
\end{equation}
where the non-negativity of the probability of this atom holds trivially. Define $z = (m_{\alpha}-m_1^{\alpha})^{1/\alpha}(\alpha-1)/\alpha$. Condition (\ref{l3}) is then implied by (\ref{l1}) since:
\begin{equation*}
\begin{array}{rll}
\displaystyle \frac{ \underline{q}(m_1,m_{\alpha},\alpha)}{\underline{q}_2} & = & \displaystyle \left(\frac{z}{m_1}\left(\frac{\alpha}{\alpha-1}\right)+\left(\frac{m_1}{z}\right)^{\alpha-1}\right)^{1/{(\alpha-1)}}, \\
& > &  \displaystyle \left(\frac{z}{m_1}+\left(\frac{m_1}{z}\right)^{\alpha-1}\right)^{1/{(\alpha-1)}}, \\
& > & 1,
\end{array}
\end{equation*}
where the first inequality holds since $\alpha > \alpha-1$ and the second inequality holds for all $m_1 > 0$ and $z > 0$, since either $z/m_1$ or $m_1/z$ is greater than or equal to 1 and $\alpha > 1$. This implies that  $\underline{q}(m_1,m_{\alpha},\alpha) > \underline{q}_2$.
\texitem{(d)} The final condition that we need to check for the validity of the probability distribution is to verify that the probability of the atom $0$ given by $1 - p - {(m_{\alpha} - m_1^{\alpha})}{(\alpha-1)}^{\alpha}/(\alpha^{\alpha}q^{\alpha})$ is strictly positive. Plugging in the value of $p$, this is equivalent to verifying that for the range of $q$, the following inequality holds:
\begin{equation*}
\begin{array}{lll}
  \left(1 - \left(m_{\alpha} - m_1^{\alpha}\right) \left(\displaystyle\frac{\alpha-1}{{\alpha}q}\right)^{\alpha}\right)^{\alpha-1} & > & \left(1 - \left(\displaystyle\frac{m_{\alpha} - m_1^{\alpha}}{m_1}\right) \left(\displaystyle\frac{\alpha-1}{{\alpha}q}\right)^{\alpha-1}\right)^{\alpha}?
\end{array}
\end{equation*}
or equivalently:
\begin{equation*}
\begin{array}{lll}
 1 - \left(m_{\alpha} - m_1^{\alpha}\right) \left(\displaystyle\frac{\alpha-1}{{\alpha}q}\right)^{\alpha}
  & > & \left(1 - \left(\displaystyle\frac{m_{\alpha} - m_1^{\alpha}}{m_1}\right) \left(\displaystyle\frac{\alpha-1}{{\alpha}q}\right)^{\alpha-1}\right)^{\alpha/(\alpha-1)}?
\end{array}
\end{equation*}
This condition is implied by (\ref{l1}) since:
\begin{equation*}
\begin{array}{rll}
\displaystyle 1 - \left(m_{\alpha} - m_1^{\alpha}\right) \left(\displaystyle\frac{\alpha-1}{{\alpha}q}\right)^{\alpha}  & = & \displaystyle 1 - \left(\displaystyle\frac{m_{\alpha} - m_1^{\alpha}}{m_1}\right) \left(\displaystyle\frac{\alpha-1}{{\alpha}q}\right)^{\alpha-1}\left(\frac{m_1(\alpha-1)}{\alpha q}\right), \\
& > & \displaystyle 1 - \left(\displaystyle\frac{m_{\alpha} - m_1^{\alpha}}{m_1}\right) \left(\displaystyle\frac{\alpha-1}{{\alpha}q}\right)^{\alpha-1}, \\
& > &  \left(1 - \left(\displaystyle\frac{m_{\alpha} - m_1^{\alpha}}{m_1}\right) \left(\displaystyle\frac{\alpha-1}{{\alpha}q}\right)^{\alpha-1}\right)^{\alpha/(\alpha-1)},
\end{array}
\end{equation*}
where the first inequality holds under the condition that $q > m_1(\alpha-1)/\alpha$, which is implied by (\ref{l1}) as $\underline{q}(m_1,m_{\alpha},\alpha) > m_1 > m_1(\alpha-1)/\alpha$ and the second inequality holds since the term in the brackets is strictly less than $1$ for $q > \underline{q}(m_1,m_{\alpha},\alpha)$ and the exponent is greater than $1$. \\
This implies that the distribution is feasible in $\mathcal{F}_{1,\alpha}$ for large values of $q$.  This leads to the desired result.
}
\end{proof}

\subsection{\textcolor{black}{Upper Bound for Large Values of $q$}}
To develop the upper bound on the worst-case expected value, we consider the dual formulation for the moment problem. We will show through an appropriate construction of a dual feasible solution in conjunction with the primal feasible distribution, that this bound is approximately optimal for large values of $q$.
\begin{prop}\label{prop2}
\textcolor{black}{Consider the ambiguity set $\mathcal{F}_{1,\alpha}$.
\texitem{(a)} When $\alpha \in (2,\infty)$, the following upper bound is valid:
\begin{equation}\label{eq:prop2}
\begin{array}{rlllll}
\displaystyle\sup_{{F} \in \mathcal{F}_{1,\alpha}}\  \mathbb{E}_{{F}}[\tilde{d}-q]_+ & \leq  &\displaystyle\frac{(m_{\alpha} - m_1^{\alpha})}{\alpha^{\alpha} q^{\alpha-1} - \alpha^2m_1^{\alpha-1}(\alpha-1)^{\alpha-1}} (\alpha-1)^{\alpha-1}, & \forall q > \overline{q}(m_1,\alpha),
 \end{array}
 \end{equation}
for all values of $q$ greater than $\overline{q}(m_1,\alpha)$ where:
\begin{equation}\label{eq:root1}
\begin{array}{rlllll}
\displaystyle  \overline{q}(m_1,\alpha)& = & \displaystyle m_1(\alpha-1)\alpha^{(2-\alpha)/(\alpha-1)}.
  \end{array}
  \end{equation}
 \texitem{(b)} When $\alpha \in (1,2)$, for all $\epsilon \in (0,(\alpha/(\alpha-1))^{\alpha-1}-\alpha)$, the following upper bound is valid:
 \begin{equation}\label{eq:prop2}
\begin{array}{rlllll}
\displaystyle\sup_{{F} \in \mathcal{F}_{1,\alpha}}\  \mathbb{E}_{{F}}[\tilde{d}-q]_+ & \leq  &\displaystyle\frac{(m_{\alpha} - m_1^{\alpha})}{\alpha^{\alpha} q^{\alpha-1} - (\alpha+\epsilon)\alpha m_1^{\alpha-1}(\alpha-1)^{\alpha-1}} (\alpha-1)^{\alpha-1}, & \forall q > \overline{q}(m_1,\alpha,\epsilon).
 \end{array}
 \end{equation}
 for all values of $q$ greater than $\overline{q}(m_1,\alpha,\epsilon) = m_1(\alpha-1)x^{*}/\alpha$ where $x^{*}$ is defined as the unique root in the interval $((\alpha+\epsilon)^{1/(\alpha-1)},\infty)$ to the equation:
\begin{equation}\label{eq:root2}
\begin{array}{rlllll}
\displaystyle   x^{\alpha} - \left(\alpha+\epsilon\right)x +1 - \left(x^{\alpha-1} -\alpha-\epsilon+1\right)^{\alpha/(\alpha-1)} & = & 0.
  \end{array}
  \end{equation}}
\end{prop}
\begin{proof}
 We derive the upper bound by constructing a dual feasible solution to (\ref{dualf}) as follows. Define $y_0,~y_1$ and $y_{\alpha}$ as:
\begin{equation}\label{dual}
\begin{array}{rcl}
y_0 & = & \displaystyle\frac{(\alpha-1)m_1^{\alpha} (\alpha-1)^{\alpha-1}}{\alpha^{\alpha} (q^{\alpha-1} - K)},\\
y_1 & = & \displaystyle\frac{-\alpha m_1^{\alpha-1} (\alpha-1)^{\alpha-1}}{\alpha^{\alpha} (q^{\alpha-1} - K)},\\
y_{\alpha} & = & \displaystyle\frac{(\alpha-1)^{\alpha-1}}{\alpha^{\alpha} (q^{\alpha-1} - K)},
\end{array}
\end{equation}
where we choose a strictly positive $K$ in a manner to be specified later in the proof.
We first verify that this forms a dual feasible solution for $q$ satisfying $q > K^{1/(\alpha-1)}$, by checking each of the dual constraints. Observe that the dual feasibility constraints are equivalent to the following conditions:
\begin{equation}
\displaystyle\min_{w \geq 0} \left(y_0 + y_1 w + y_{\alpha} w^{\alpha}\right) \geq 0 \mbox{ and } \displaystyle\min_{w \geq 0} \left(y_0 + q + (y_1 - 1) w + y_{\alpha} w^{\alpha} \right) \geq 0\end{equation}
Since the values of the dual variables in (\ref{dual}) satisfy $y_{\alpha} > 0$ and $y_1 < 0$ for $q > K^{1/(\alpha-1)}$, the minimum value in the first dual constraint is obtained at $w^* = (\displaystyle{-y_1}/{(\alpha y_{\alpha})})^{1/(\alpha-1)}$. Substituting this value, the first dual feasibility constraint is equivalent to verifying the condition:
\begin{equation*}
\begin{array}{rlll}
y_0 & \geq & -y_1  \left(\displaystyle\frac{-y_1}{\alpha y_{\alpha}}\right)^{1/(\alpha-1)} - y_{\alpha}  \left(\displaystyle\frac{-y_1}{\alpha y_{\alpha}}\right)^{\alpha/(\alpha-1)},\\
& = & \displaystyle\frac{(-y_1)^{\alpha/(\alpha-1)}}{(\alpha y_{\alpha})^{1/(\alpha-1)}} \left(\displaystyle\frac{\alpha-1}{\alpha}\right).
\end{array}
\end{equation*}
The choice of dual variables in (\ref{dual}) satisfy this condition at equality since:
\begin{equation*}
\begin{array}{rlll}
\displaystyle y_0 - \frac{(-y_1)^{\alpha/(\alpha-1)}}{(\alpha y_{\alpha})^{1/(\alpha-1)}} \left(\displaystyle\frac{\alpha-1}{\alpha}\right)& = &  \displaystyle\frac{(\alpha-1)m_1^{\alpha} (\alpha-1)^{\alpha-1}}{\alpha^{\alpha} (q^{\alpha-1} - K)}  - \left(\frac{\alpha-1}{\alpha}\right)\frac{({\alpha}m_1^{\alpha-1}(\alpha-1)^{\alpha-1})^{\alpha/(\alpha-1)}}{(\alpha(\alpha-1)^{\alpha-1})^{1/(\alpha-1)}\alpha^{\alpha} (q^{\alpha-1} - K)},\\
& = & 0.
\end{array}
\end{equation*}
Furthermore as $y_{\alpha} > 0$, the minimum value in the second dual constraint is obtained at:
 $$w^* = (\displaystyle{(1-y_1)}/{(\alpha y_{\alpha})}
)^{1/(\alpha-1)}.$$  Substituting this in, the second dual feasibility constraint is equivalent to verifying if the following condition holds:
\begin{equation}
\begin{array}{rlll}
\delta(q) & := & y_0 + q - \displaystyle\frac{(1-y_1)^{\alpha/(\alpha-1)}}{(\alpha y_{\alpha})^{1/(\alpha-1)}} \left(\displaystyle\frac{\alpha-1}{\alpha}\right) \geq 0?
\end{array}
\end{equation}
The choice of dual variables in (\ref{dual}) leads to the following expression:
\begin{equation*}
\begin{array}{rllllll}
\displaystyle \delta(q) & = & \displaystyle \frac{(\alpha-1)m_1^{\alpha} (\alpha-1)^{\alpha-1}}{\alpha^{\alpha} (q^{\alpha-1} - K)} + q - \frac{\left(\alpha^{\alpha} (q^{\alpha-1} - K) + n m_1^{\alpha-1} (\alpha-1)^{\alpha-1}\right)^{\alpha/(\alpha-1)}}{\left(\alpha(\alpha-1)\right)^{1/(\alpha-1)}\alpha^{\alpha} (q^{\alpha-1} - K)}\left(\displaystyle\frac{\alpha-1}{\alpha}\right),\\
& = & \displaystyle \frac{m_1^{\alpha} {(\alpha-1)}^{\alpha} + \alpha^{\alpha} q (q^{\alpha-1} - K) - \left(\alpha^{\alpha-1} (q^{\alpha-1} - K) +  m_1^{\alpha-1} (\alpha-1)^{\alpha-1} \right)^{\alpha/(\alpha-1)}}{\alpha^{\alpha} (q^{\alpha-1} - K)}.
\end{array}
\end{equation*}
\textcolor{black}{We verify that for $q > K^{1/(\alpha-1)}$ (to ensure positivity of the denominator) and large enough, the following inequality holds:
\begin{equation*}
\begin{array}{rcl}
 m_1^{\alpha} {(\alpha-1)}^{\alpha} + \alpha^{\alpha} q (q^{\alpha-1} - K) - \left(\alpha^{\alpha-1} (q^{\alpha-1} - K) +  m_1^{\alpha-1} (\alpha-1)^{\alpha-1} \right)^{\alpha/(\alpha-1)} >  0?,
\end{array}
\end{equation*}
Let $C = m_1(\alpha-1)/{\alpha} > 0$. Dividing by $m_1^\alpha(\alpha-1)^{\alpha}$, this condition is equivalent to verifying that for $q$ large enough, the following inequality holds:
\begin{equation*}
\begin{array}{rcl}
 \triangle(q) & := & \displaystyle \left(\left(\frac{q}{C}\right)^{\alpha} - q \left(\frac{K}{C^{\alpha}}\right) +1 \right) - \left(\left(\frac{q}{C}\right)^{\alpha-1} - \left(\frac{K}{C^{\alpha-1}}\right) +1\right)^{\alpha/(\alpha-1)} >  0?
\end{array}
\end{equation*}
We need to verify that $\triangle(q)$ is strictly positive for large values of $q$. To do so, we consider two cases:
\texitem{(a)} $\alpha \in (2,\infty)$:
 In this case, we set the constant $K$ to ${\alpha}C^{\alpha-1} = \alpha(m_1(\alpha-1)/{\alpha})^{\alpha-1}$. Then, we need to verify that
 for $q$ beyond a certain value $\overline{q}(m_1,m_{\alpha},\alpha)$ (which will be identified next), the following inequality holds:
\begin{equation*}
\begin{array}{rcllll}
\triangle(q) & = &\displaystyle \left(\left(\frac{q}{C}\right)^{\alpha} - \alpha \left(\frac{q}{C}\right) +1 \right) - \left(\left(\frac{q}{C}\right)^{\alpha-1} - \alpha +1\right)^{\alpha/(\alpha-1)} > 0?
\end{array}
\end{equation*}
By setting $\overline{q}(m_1,m_{\alpha},\alpha) = K^{1/(\alpha-1)} = \alpha^{1/(\alpha-1)} C = m_1(\alpha-1)\alpha^{(2-\alpha)/(\alpha-1)}$, we observe that the condition is satisfied at equality, since:
\begin{equation} \label{aaaaa0}
\begin{array}{rllll}
 \triangle(\overline{q}(m_1,m_{\alpha},\alpha)) & = & \displaystyle   \left(\alpha^{\alpha/(\alpha-1)} - \alpha^{\alpha/(\alpha-1)} +  1\right)- \left(\alpha - \alpha +1 \right)^{\alpha/(\alpha-1)}, \\
 & = &\displaystyle  0.
\end{array}
\end{equation}
Furthermore, the derivative of the function $\triangle({q})$ with respect to $q$ satisfies:
\begin{equation} \label{aaaaa}
\begin{array}{rllll}
 \displaystyle \frac{d}{dq}\triangle({q}) & = & \displaystyle \frac{\alpha}{C}  \left(\left(\frac{q}{C}\right)^{\alpha-1}-1-\left(\frac{q}{C}\right)^{\alpha-2}\left(\left(\frac{q}{C}\right)^{\alpha-1}-\alpha+1\right)^{1/(\alpha-1)}\right),\\
 & = & \displaystyle  \frac{\alpha q^{\alpha-1}}{C^{\alpha}}\left(1 - \frac{1}{(q/C)^{\alpha-1}} - \left(1-\frac{\alpha-1}{(q/C)^{\alpha-1}}\right)^{1/(\alpha-1)} \right), \\
 & > & 0, \quad \quad \forall q > \overline{q}(m_1,m_{\alpha},\alpha),
\end{array}
\end{equation}
where the first equality is obtained by differentiating the function $\triangle({q})$, the second equality is obtained by straightforward algebraic manipulations and the inequality is obtained by using Bernoulli's inequality $(1-x)^t > 1-tx$ which is valid for $t > 1$ and $0 < x \leq 1$ and setting $t = \alpha -1$ and $x = (C/q)^{\alpha-1}$. Note that since $\alpha > 2$, $q > \alpha^{1/(\alpha-1)} C \geq C$ and the conditions $t > 1$ and $0 \leq x \leq 1$ are satisfied. This implies that the derivative of the function is positive for all values of $q > \overline{q}(m_1,m_{\alpha},\alpha)$. Combining (\ref{aaaaa0}) and (\ref{aaaaa}) implies that $\triangle({q})$ is positive for all values of $q$ above the threshold:
\begin{equation*}
\begin{array}{rllll}
 \triangle(q) & > & \displaystyle   0, & \displaystyle \forall q > \overline{q}(m_1,m_{\alpha},\alpha).
\end{array}
\end{equation*} Hence the constructed solution is dual feasible for $q$ above the $\overline{q}(m_1,m_{\alpha},\alpha)$. The objective function value of this dual feasible solution reduces to the form below which yields the desired result:
 \begin{equation*}
 \begin{array}{rllll}
\displaystyle y_0 + y_1 m_1 + y_{\alpha} m_{\alpha} & = & \displaystyle\frac{(m_{\alpha} - m_1^{\alpha}) (\alpha-1)^{\alpha-1}}{\alpha^{\alpha} q^{\alpha-1} - \alpha^2m_1^{\alpha-1}(\alpha-1)^{\alpha-1}}.
\end{array}
\end{equation*}
\texitem{(b)} $\alpha \in (1,2)$: Note that unlike the $\alpha > 2$ case, setting $K = {\alpha}C^{\alpha-1}$ does not ensure a dual feasible solution for $\alpha \in (1,2)$. To see this, observe that by applying the generalized binomial expansion, the term $\triangle(q)$ reduces to:
\begin{equation*}
\begin{array}{rllll}
\triangle(q)& = & \displaystyle \left(\left(\frac{q}{C}\right)^{\alpha} - \alpha \left(\frac{q}{C}\right) +1 \right) - \sum_{k=0}^{\infty} {{\frac{\alpha}{\alpha-1}}\choose{k}}\left(\frac{q}{C}\right)^{\alpha-(\alpha-1)k}(1-\alpha)^k,
 \end{array}
 \end{equation*}
 where ${r}\choose{k}$ is defined as $r(r-1)\ldots(r-k+1)/k!$ for general values of $r$ (not necessarily integer). Expanding the first few terms, gives:
 \begin{equation*}
\begin{array}{rllll}
\triangle(q) & = &
  \displaystyle \displaystyle \left(\frac{q}{C}\right)^{\alpha} - \alpha \left(\frac{q}{C}\right) +1 - \left(\frac{q}{C}\right)^{\alpha}+ \alpha \left(\frac{q}{C}\right)-\frac{\alpha}{2} \left(\frac{q}{C}\right)^{2-\alpha}-\sum_{k=3}^{\infty} {{\frac{\alpha}{\alpha-1}}\choose{k}}\left(\frac{q}{C}\right)^{\alpha-(\alpha-1)k}(1-\alpha)^k,\\
  & =& \displaystyle -\frac{\alpha}{2} \left(\frac{q}{C}\right)^{2-\alpha} +1 -\sum_{k=3}^{\infty} {{\frac{\alpha}{\alpha-1}}\choose{k}}\left(\frac{q}{C}\right)^{\alpha-(\alpha-1)k}(1-\alpha)^k,\\ 
 \end{array}
 \end{equation*}
where the leading term of the expression with exponent $2-\alpha > 0$ has a negative coefficient. This implies that for large values of $q$, $\triangle(q)$ becomes negative. To deal with this, we modify the dual solution by choosing for a strictly positive small $\epsilon > 0$, the value $K = ({\alpha}+\epsilon)C^{\alpha-1}$. In this case, we need to verify that for $q$ above a certain value (that needs to be identified), the following inequality holds:
\begin{equation*}
\begin{array}{rcllll}
\triangle(q) & = &\displaystyle \left(\left(\frac{q}{C}\right)^{\alpha} - (\alpha+\epsilon) \left(\frac{q}{C}\right) +1 \right) - \left(\left(\frac{q}{C}\right)^{\alpha-1} - \alpha-\epsilon +1\right)^{\alpha/(\alpha-1)} > 0?
\end{array}
\end{equation*}
Note that, by applying the generalized binomial expansion, the term $\triangle(q)$ reduces to:
\begin{equation*}
\begin{array}{rllll}
\triangle(q)& = & \displaystyle \left(\frac{q}{C}\right)^{\alpha} - (\alpha+\epsilon) \left(\frac{q}{C}\right) +1  - \sum_{k=0}^{\infty} {{\frac{\alpha}{\alpha-1}}\choose{k}}\left(\frac{q}{C}\right)^{\alpha-(\alpha-1)k}(1-\alpha-\epsilon)^k, \\
 &= & \displaystyle \left(\frac{q}{C}\right)^{\alpha} - (\alpha+\epsilon) \left(\frac{q}{C}\right) +1 -\left(\frac{q}{C}\right)^{\alpha} +\alpha\left(1+\frac{\epsilon}{\alpha-1}\right)\left(\frac{q}{C}\right) \\
 & & \displaystyle - \sum_{k=2}^{\infty} {{\frac{\alpha}{\alpha-1}}\choose{k}}\left(\frac{q}{C}\right)^{\alpha-(\alpha-1)k}(1-\alpha-\epsilon)^k,\\
 & = & \displaystyle \left(\frac{\epsilon}{\alpha-1}\right)\left(\frac{q}{C}\right) + 1 - \sum_{k=2}^{\infty} {{\frac{\alpha}{\alpha-1}}\choose{k}}\left(\frac{q}{C}\right)^{\alpha-(\alpha-1)k}(1-\alpha-\epsilon)^k,\\
 \end{array}
 \end{equation*}
 where the leading term of the expression with exponent $1$ has a positive coefficient.
The derivative of the function $\triangle({q})$ with respect to $q$ is given by:
\begin{equation} \label{aaaaab}
\begin{array}{rllll}
 \displaystyle \frac{d}{dq}\triangle({q}) & = & \displaystyle \frac{\alpha}{C}  \left(\left(\frac{q}{C}\right)^{\alpha-1}-1-\frac{\epsilon}{\alpha}-\left(\frac{q}{C}\right)^{\alpha-2}\left(\left(\frac{q}{C}\right)^{\alpha-1}-\alpha-\epsilon+1\right)^{1/(\alpha-1)}\right).
\end{array}
\end{equation}
Furthermore, the second derivative of the function $\triangle({q})$ with respect to $q$ is given by:
\begin{equation} \label{aaaaab1}
\begin{array}{rllll}
 \displaystyle \frac{d^2}{dq^2}\triangle({q}) & = & \displaystyle \frac{\alpha q^{\alpha-3}}{C^{\alpha-1}} h(q),\\
\end{array}
\end{equation}
where $h(q)$ is defined as:
\begin{equation} \label{aaaaab2}
\begin{array}{rllll}
 \displaystyle h({q}) & := & \displaystyle  (\alpha-1)\left(\frac{q}{C}\right) + (2-\alpha)\left(\left(\frac{q}{C}\right)^{\alpha-1}-\alpha-\epsilon+1\right)^{1/(\alpha-1)} \\
 & & \displaystyle -\left(\frac{q}{C}\right)^{\alpha-1}\left(\left(\frac{q}{C}\right)^{\alpha-1}-\alpha-\epsilon+1\right)^{(2-\alpha)/(\alpha-1)}, \\
 & > &  0, \quad \quad \quad \forall q > (\alpha+\epsilon-1)^{1/(\alpha-1)}C,
\end{array}
\end{equation}
with the nonnegativity of the second derivative following from using the strict form of the weighted arithmetic and geometric mean inequality given by $\lambda x_1 + (1-\lambda) x_2 > x_1^{\lambda}x_2^{1-\lambda}$ which is valid for $\lambda \in (0,1)$ and $x_1, x_2> 0$, $x_1 \neq x_2$ by setting $\lambda = \alpha -1 \in (0,1)$ with $\alpha \in (1,2)$, $x_1 =  q/C$ and $x_2 = ((q/C)^{\alpha-1}-\alpha-\epsilon+1)^{1/(\alpha-1)}$. To finish the proof, observe $\underline{q} = (\alpha+\epsilon)^{1/(\alpha-1)}C$ is a root of the equation $ \triangle(q) = 0$, since:
\begin{equation} \label{aaaaa011}
\begin{array}{rllll}
 \triangle(\underline{q} ) & = & \displaystyle  \left(\left(\alpha+\epsilon\right)^{\alpha/(\alpha-1)} - (\alpha+\epsilon)^{\alpha/(\alpha-1)}  +1 \right) - \left(\alpha+\epsilon - \alpha-\epsilon +1\right)^{\alpha/(\alpha-1)} , \\
 & = &\displaystyle  0.
\end{array}
\end{equation}
Furthermore, the derivative:
\begin{equation} \label{aaaaa0111}
\begin{array}{rllll}
\displaystyle \frac{d}{dq} \triangle(q)|_{q = \underline{q}} & = & \displaystyle  \frac{\alpha}{C}  \left(\alpha+\epsilon-1-\frac{\epsilon}{\alpha}-\left(\alpha+\epsilon\right)^{(\alpha-2)/(\alpha-1)}\left(\alpha+\epsilon-\alpha-\epsilon+1\right)^{1/(\alpha-1)}\right), \\
& = & \displaystyle \frac{\alpha+\epsilon}{C}  \left(\alpha-1-\alpha\left(\alpha+\epsilon\right)^{-1/(\alpha-1)}\right), \\
 & < &\displaystyle  0,
\end{array}
\end{equation}
where the first equation is obtained by plugging in $\underline{q}$ into (\ref{aaaaab}), the second equation is obtained from straightforward algebraic manipulations and the inequality is obtained by observing that for $0 < \epsilon < (\alpha/(\alpha-1))^{\alpha-1}-\alpha$, the right hand side is negative. Since the function is strictly convex from (\ref{aaaaab2}) with $\lim_{q \rightarrow \infty}\triangle(q) = \infty$ and one of the roots is given by $\underline{q}$ where the derivative is negative, the function is positive for all values of $q$ above the second root $\overline{q}$ to the equation:
\begin{equation*}
\begin{array}{rlllll}
\displaystyle  \triangle(\overline{q})& = & 0,
  \end{array}
  \end{equation*}
  which lies in the range $(\underline{q},\infty)$. The objective function value of this dual feasible solution reduces to the form below which yields the desired result for $\alpha \in (1,2)$:
 \begin{equation*}
y_0 + y_1 m_1 + y_{\alpha} m_{\alpha} = \displaystyle\frac{(m_{\alpha} - m_1^{\alpha}) (\alpha-1)^{\alpha-1}}{\alpha^{\alpha} q^{\alpha-1} - (\alpha+\epsilon)\alpha m_1^{\alpha-1}(\alpha-1)^{\alpha-1}}.
\end{equation*}
}
\end{proof}

\subsection{Numerical Example}\label{subsec:3.3}
We provide a numerical illustration of the quality of the bounds from Propositions \ref{prop1} and \ref{prop2} respectively. To compute the worst-case expected value, we solve the dual formulation in (\ref{dualf}) using a semidefinite program (SDP) for rational values of $\alpha$. Assume that $\alpha = {p}/{q}$, where $p$ and $q$ are strictly positive integers. Then, the dual formulation is given as:
\begin{equation} \label{dualf1}
\begin{array}{rlll}
\inf & y_0 + y_1 m_1 + y_{p/q} m_{p/q} &\\
\textrm{s.t.} & y_0 + y_1 w + y_{p/q} w^{p/q} \geq 0,  & \forall w \geq 0,\\
& y_0+q + (y_1-1) w + y_{p/q} w^{p/q} \geq 0, & \forall w \geq 0.
\end{array}
\end{equation}
Applying the transformation by defining the variable $d = w^{1/q}$ or equivalently $d = e^q$, we obtain a reformulation of the dual problem as:
\begin{equation} \label{dualf2}
\begin{array}{rlll}
\inf & y_0 + y_1 m_1 + y_{p/q} m_{p/q} &\\
\textrm{s.t.} & y_0 + y_1 d^q + y_{p/q} d^{p} \geq 0,  & \forall d \geq 0,\\
& y_0+q + (y_1-1) d^q + y_{p/q} d^{p} \geq 0, & \forall d \geq 0,
\end{array}
\end{equation}
The constraints in (\ref{dualf2}) are the standard nonnegativity conditions on univariate polynomials over the half-line for which semidefinite representations are available (see Bertsimas and Popescu \cite{bertsimas02}, Lasserre \cite{jean02}, Nesterov \cite{Nesterov}). For example, with $\alpha = 3$ ($p = 3, q = 1$), the semidefinite programming formulation is given as:
\begin{equation} \label{dualf3}
\begin{array}{rlll}
\displaystyle \inf_{y_0,y_1,y_3,a_1,b_1,c_1,a_2,b_2,c_2} & y_0 + y_1 m_1 + y_{3} m_{3} &\\
\textrm{s.t.} & \begin{bmatrix}
    y_0& 0 &a_1& b_1\\
   0 & y_1-2a_1& -b_1 & c_1 \\
   a_1 & -b_1 & -2c_1 & 0 \\
   b_1& c_1 & 0 & y_3
  \end{bmatrix}\succeq 0 ,  & \\
& \begin{bmatrix}
    y_0+q& 0 &a_2 & b_2\\
   0 & y_1-1-2a_2&-b_2 & c_2 \\
   a_2 & -b_2 & -2c_2 & 0 \\
   b_2& c_2 & 0 & y_3
  \end{bmatrix}\succeq 0 , &
\end{array}
\end{equation}
while for $\alpha = 3/2$ ($p = 3, q = 2$), the semidefinite programming formulation is given as:
\begin{equation} \label{dualf4}
\begin{array}{rlll}
\displaystyle \inf_{y_0,y_1,y_{3/2},a_1,b_1,c_1,a_2,b_2,c_2} & y_0 + y_1 m_1 + y_{3/2} m_{3/2} &\\
\textrm{s.t.} & \begin{bmatrix}
    y_0& 0 &a_1 & b_1\\
   0 & -2a_1& -b_1 & c_2 \\
   a_1 & -b_1 & y_1-2c_2 & 0 \\
   b_1& c_2 & 0 & y_{3/2}
  \end{bmatrix}\succeq 0 ,  & \\
& \begin{bmatrix}
    y_0+q& 0 &a_1 & b_1\\
   0 & -2a_1& -b_1 & c_2 \\
   a_1 & -b_1 & y_1-1-2c_2 & 0 \\
   b_1& c_2 & 0 & y_{3/2}
  \end{bmatrix}\succeq 0 , &
\end{array}
\end{equation}
In Figures \ref{m31} and \ref{m32}, we compare the upper and lower bounds and the worst-case expected value obtained from solving the SDP. The semidefinite programs were solved in Matlab R2017a with SDPT3 version 4.0 (see Toh, Todd and Tutuncu \cite{toh1,toh2}). \textcolor{black}{To compare the results, we use a mean demand of $50$ and assume that $m_{3} = 125150$ and $m_{3/2} = \sqrt{125150}$ respectively where Holder's inequality requires $m_{3} \geq m_{3/2}^2 \geq m_1^3$. The value of the worst-case expected values and the bounds for $\alpha = 3$ are smaller than the values for $\alpha = 3/2$ for a given $q$ as should be expected since the former ambiguity set makes stronger assumptions on the existence of moments. The figures also illustrate that the scaling behavior of the bounds as a function of $q$ and provides the range of $q$ from Propositions \ref{prop1} and \ref{prop2} over which the bounds are valid in these instances. We observe that for $\alpha = 3/2$ for the range of $q$ considered in the figure, the upper bounds are closer to the exact value in comparison to lower bound. While this suggests that it might be possible to construct stronger closed form lower bounds, especially when $\alpha < 2$, as we show in the next section, the proposed lower and upper bounds are sufficient to provide a characterization of the worst-case value for large values of $q$ using the theory of regularly varying functions.}
\begin{figure}[htbp]
\centering
\includegraphics[width=15cm] {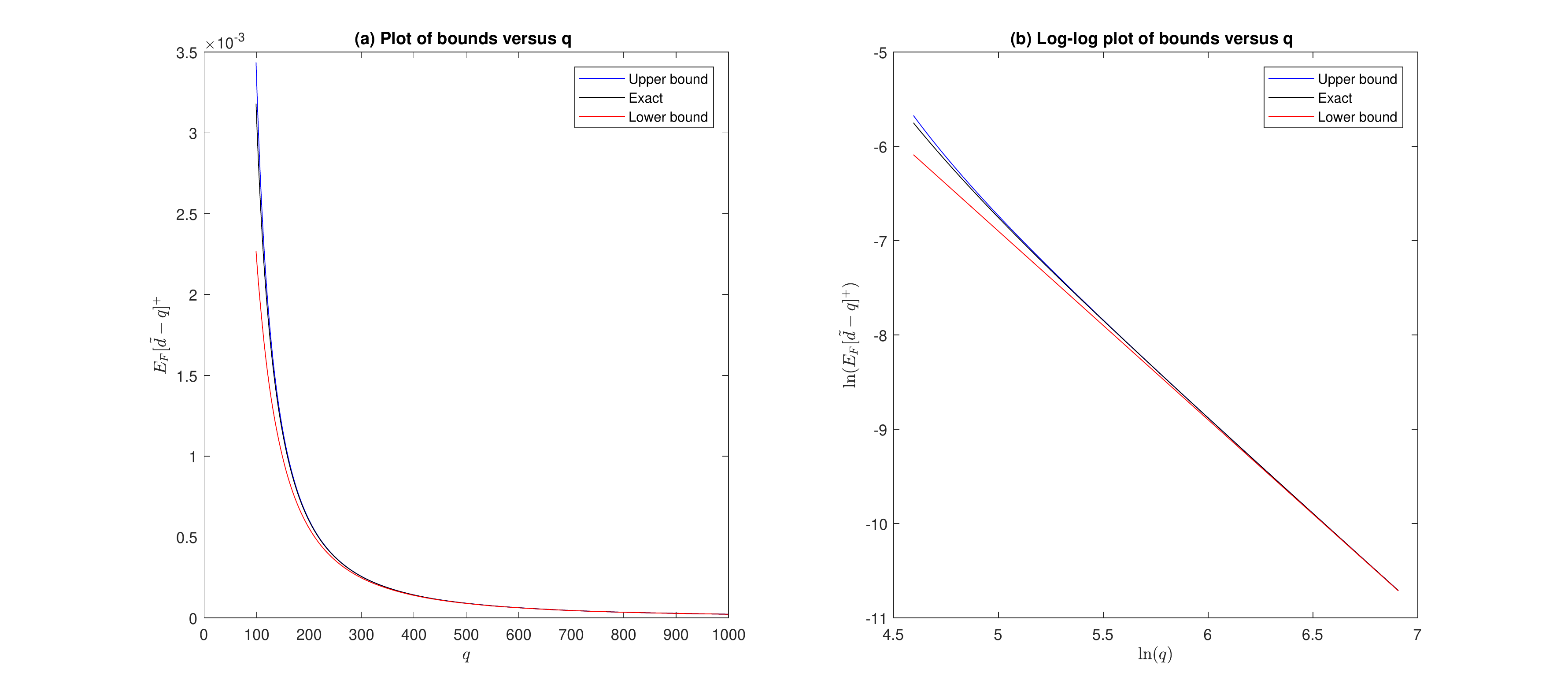}
\caption{Plot (a) compares the upper and lower bounds with the exact bound obtained from solving a SDP as a function of $q$ while plot (b) provides a log-log plot to characterize the scaling behavior. The mean demand is set to $m_1 = 50$ and the third moment is set to $m_3 = 125150$. \textcolor{black}{The lower bound is valid for $q > 50.013$ and the upper bound is valid for $q > 57.735$.}}
\label{m31}
\end{figure}

\begin{figure}[htbp]
\centering
\includegraphics[width=15cm] {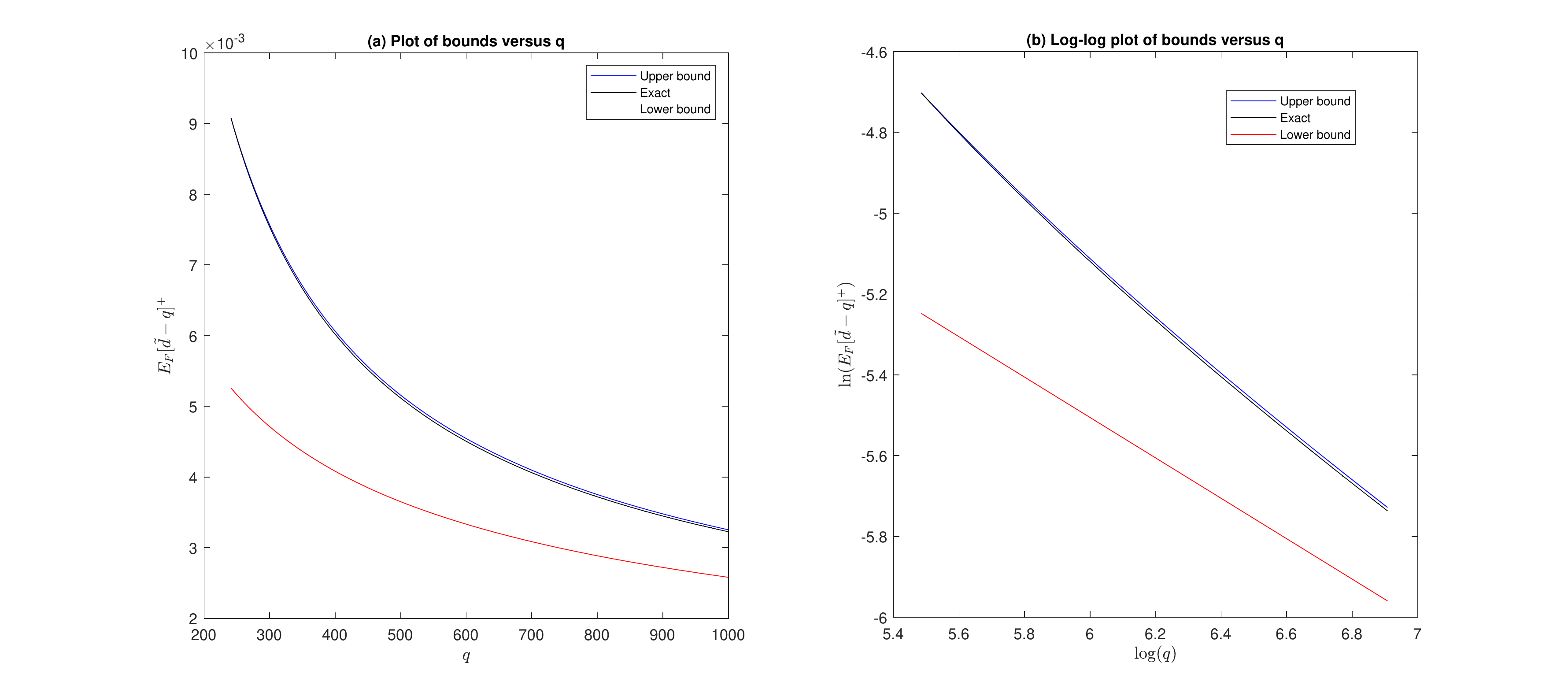}
\caption{\textcolor{black}{Plot (a) compares the upper and lower bounds with the exact bound obtained from solving a SDP as a function of $q$ while plot (b) provides a log-log plot to characterize the scaling behavior. The mean demand is set to $m_1 = 50$ and the highest moment is set to $m_{3/2} = \sqrt{125150}$ with $\alpha = 3/2$. The bounds are plotted by setting $\epsilon = 0.1$ where the lower bound is valid for $q > 50.034$ and the upper bound is valid for $q > 240.67$.}} \label{m32}
\end{figure}
\section{Characterization of Heavy-Tail Optimality}\label{sec4}
In this section, we use the lower and upper bounds to provide a characterization of the tail of the demand distribution $F^*$ for which the distributionally robust newsvendor order quantity remains optimal. To do so, we make use of the notion of regularly varying distributions that is popularly used to characterize heavy-tailed distributions (see Bingham, Goldie and Teugels \cite{bingham}, de Haan \cite{dehaan}). The key property of such distributions is that the behaviour at infinity is similar to the behaviour of a power law distribution. As we see in this section, this is exactly the type of behavior that $F^*$ satisfies.

\subsection{Regularly Varying Distributions}
We first review the popular paradigm of distributions with regularly varying tails that has been used to characterize non-negative heavy-tailed distributions. A function $u:\Re_{+}\to\Re_{+}$ is said to be regularly varying at infinity with index $\alpha\in \Re$ if for all $t > 0$, we have:
\begin{equation*}
\begin{array}{rllll}
\displaystyle \lim_{x\to\infty} \frac{u(tx)}{u(x)} = t^{\alpha}.
\end{array}
\end{equation*}
 We express this by $u\in \RV_{\alpha}$. A non-negative random variable $\tilde{d}$ with cumulative distribution function $F$ is regularly varying if $\overline{F}:=1-F \in \RV_{-\alpha}$ for some $\alpha\ge 0$. The distribution function is said to have tail parameter $\alpha$ if $\overline{F} \in \RV_{-\alpha}$.
Two classical examples of regularly varying random variables that are particularly relevant in our context are:
\texitem{(a)} $\mbox{Pareto}(x_m,\alpha)$: This random variable is defined with two parameters - a scale parameter $x_m > 0$ and a shape parameter $\alpha > 0$ with probability density function given as follows:
\begin{equation}\label{pareto}
g(w) = \displaystyle \frac{\alpha x_m^{\alpha}}{w^{\alpha+1}}, ~\forall w \geq x_m,
\end{equation}
Then for $w \geq x_{m}$, we have:
\[ \overline{F}(w):= \int_{w}^{\infty} g(x) \mathrm d x =  x_m^{\alpha}w^{-\alpha},\]
and hence clearly $\overline{F} \in \RV_{-\alpha}$. Note that in Grundy's model discussed in Section \ref{lb}, we obtain a characterization of the distribution $F^{*}$ as follows:
\begin{equation} \label{fff}
\displaystyle {F}^*(w) =  \mathbb{P}(\tilde{d}^* \leq w) = \left\{\begin{array}{ll}
\displaystyle  \displaystyle1-m_{\alpha} \left(\displaystyle\frac{\alpha-1}{{\alpha}w}\right)^{\alpha}, & \textrm{if } \displaystyle w > \displaystyle\frac{\alpha-1}{\alpha} m_{\alpha}^{1/\alpha},\\
\displaystyle 0 & \textrm{if }\displaystyle 0 \leq w \leq \displaystyle\frac{\alpha-1}{\alpha} m_{\alpha}^{1/\alpha}. \end{array}\right.
\end{equation}
This defines a Pareto random variable as follows:
\begin{equation} \label{g}
\displaystyle \tilde{d}^* = \mbox{Pareto}\left(\frac{(\alpha-1)m_{\alpha}^{1/\alpha}}{\alpha},{\alpha}\right),
\end{equation}
where $\overline{F}^* \in \RV_{-\alpha}$.
\texitem{(b)} $\tilde{t}_{\nu}(\mu,\sigma^{2})$: The t-location scale random variable is defined with three parameters - a location parameter $\mu > 0$ and a scale parameter $\sigma > 0$ and degree of freedom parameter $\nu$. This distribution is regularly varying at infinity with index $\nu$. Note that in Scarf's model discussed in Section \ref{related}, we have $\overline{F}^* \in \RV_{-2}$.

Regularly varying functions have a rich theory (see Bingham, Goldie and Teugels \cite{bingham}, de Haan \cite{dehaan}) and has found many applications in the study of power-law distributions and extreme risk behavior in insurance, finance, telecommunication, social networks (see Embrechts, Kl\"u{ppelberg, Mikosch \cite{embrechts}, Resnick \cite{resnick} for details). The class of regularly varying functions admits certain nice properties with respect to summation, composition, taking quotients, integrating and differentiating which helps in understanding tail behavior of regularly varying random variables, their moments  and other functionals. The following result below attributed to Karamata \cite{karamata} shows the effect of integration on regularly varying functions. We state the special case relating to regularly varying distributions with at least first moment finite (see Resnick \cite[Theorem 0.6(a)]{resnick}). In the following theorems, one can think of $U$ as the distribution tail and $u$ as the density in the context of distribution functions.

\begin{theorem}[Karamata's Theorem] \label{thm:karamata}
Suppose $u:\Re_{+}\to \Re_{+}$ satisfies  $u\in \RV_{-\alpha}$ for some $\alpha>1$. Then $\int_{x}^{\infty} u(t) \,\mathrm dt$ is finite, $\int_{x}^{\infty} \, u(t) \mathrm dt \in \RV_{-\alpha+1}$ and:
\begin{align*}\label{lim:kar}
\lim_{x\to \infty} \frac{xu(x)}{\int_{x}^{\infty} u(t) \,\mathrm dt} = \alpha-1.
\end{align*}
\end{theorem}

The next result provides the reverse implication to Karamata's theorem and shows what happens when a regularly varying function is differentiated; see Landau \cite{landau}, Bingham, Goldie and Teugels \cite[Theorem 1.6.1]{bingham}, de Haan \cite[p. 23]{dehaan}, Resnick \cite[Theorem 0.7]{resnick} for different formulations and proofs.
\begin{theorem} \label{thm:karamatainv}
Suppose $u:\Re_{+}\to \Re_{+}$ is locally integrable in $[0,\infty)$ and define:
\[ U(x) := \int_{x}^{\infty} u(t) \, \mathrm dt.\]
\begin{enumerate}[label=(\alph*)]
\item If for $\alpha>0$, we have  functions $u$ and $U$ satisfying:
\begin{align}\label{eq:vonmises}
\lim\limits_{x\to\infty} \frac{xu(x)}{U(x)} = -\alpha ,
\end{align}
then $U\in \RV_{-\alpha}$.
\item If $U\in \RV_{-\alpha}$ for $\alpha>0$ and $u$ is monotone, then \eqref{eq:vonmises} holds and $ u \in \RV_{-\alpha-1}$.
\end{enumerate}
\end{theorem}

\subsection{\DAS{From the Ambiguity Set $\mathcal{F}_{1,\alpha}$ to a Regularly Varying Distribution $F^*$}}
Propositions \ref{prop1} and \ref{prop2} indicate that in fact the tails of the worst-case expected value are close to a power-law (Pareto-like) tail with index $\alpha-1$. In this section, we show that there exists a random variable $d^{*} \sim F^{*}$ which attains the worst-case expected value for large values of $q$ using the theory of regularly varying functions.


\begin{theorem}\label{thm1}
Given the ambiguity set $\mathcal{F}_{1,\alpha}$ as defined in \eqref{nmom}, the following holds:
\begin{enumerate}[label=(\alph*)]
\item The worst-case expected value as a function of $q$ satisfies:
$$\Pi_{1,\alpha}(q):=  \sup_{{F} \in \mathcal{F}_{1,\alpha}}  \mathbb{E}_{{F}}[\tilde{d}-q]_+ \in \RV_{-(\alpha-1)}.$$
\item There exists a distribution function $F^{*}$ that does not lie in $\mathcal{F}_{1,\alpha}$, such that for $\tilde{d}^{*} \sim F^{*}$, we have:
$$ \mathbb{E}_{{F^{*}}}[\tilde{d}^{*}-q]_+=\sup_{{F} \in \mathcal{F}_{1,\alpha}}\  \mathbb{E}_{{F}}[\tilde{d}-q]_+,  \quad \forall q \geq 0.$$
Furthermore, $\overline{F}^{*} \in \RV_{-{\alpha}}$ with $\mathbb{E}_{F^{*}} [(\tilde{d}^{*})^{\alpha_1}]<\infty$ for all $0\le \alpha_1<\alpha$ and $\mathbb{E}_{F^{*}} [(\tilde{d}^{*})^{\alpha_1}]=\infty$ if $\alpha_1\ge \alpha$.
\end{enumerate}
\end{theorem}

\begin{proof}
\DAS{Note that the case $\alpha=2$ boils down to Scarf's model  discussed in Section \ref{subsec:scarf}. From \eqref{e2}, we have for $q\ge m_{2}/2m_{1}$,
\begin{align*}
\Pi_{1,2}(q)=  \frac 12 \left(\sqrt{q^{2}-2m_{1}q+m_{2}} -(q-m_{1})\right).
\end{align*}
Now it is easy to check that:
\begin{align*}
\lim_{t\to\infty} \frac{\Pi_{1,2}(tq)}{\Pi_{1,2}(t)} = \lim_{t\to\infty} \frac{\frac{m_{2}}{2tq} + o(1/t)}{\frac{m_{2}}{2t} + o(1/t)} = \frac 1q.
\end{align*}
Hence $\Pi_{1,2}(q) \in \RV_{-1}$ which is as claimed in (a). Moreover, from \eqref{f1}-\eqref{h1}, we have $\tilde{d}^{*}\sim F^{*}$ satisfying $(b)$ where $\overline{F}^* \in \RV_{-2}$ and  is a censored $t$-distribution with shape parameter (degree of freedom) $\nu=2$ for which $\mathbb{E}_{F^*}[\tilde{d}^{*2}]=\infty$. We now concentrate on $\alpha\neq2$ for the rest of the proof.}

\noindent \DAS{\textit{(a)}  Notice that from Proposition \ref{prop1}, for $q > \underline{q}(m_1,m_{\alpha},\alpha)$, we have:
\begin{align}\label{pp1}
\Pi_{1,\alpha}(q) = \sup_{{F} \in \mathcal{F}_{1,\alpha}}\  \mathbb{E}_{{F}}[\tilde{d}-q]_+ \ge C_{1} \frac 1{q^{\alpha-1}},
\end{align}
where  $C_{1}=\left(m_{\alpha}-m_{1}^{\alpha}\right)\frac{(\alpha-1)^{\alpha-1}}{{\alpha}^{\alpha}}$.  Furthermore, we have the following upper bounds.
\begin{itemize}
\item[(i)] For $\alpha\in (2,\infty)$, using Proposition \ref{prop2}, for $q > \overline{q}(m_1,\alpha)$, we have:
\begin{align}\label{pp2a}
\Pi_{1,\alpha}(q)\le C_{1} \frac 1{q^{\alpha-1}} \left(1- \frac{C_{2}}{q^{\alpha-1}}\right)^{-1},
\end{align}
where $C_{2}= m_{1}^{\alpha-1}\frac{(\alpha-1)^{\alpha-1}}{\alpha^{\alpha-2}}$.
\item[(ii)] For $\alpha\in (1,2)$, fixing small $\epsilon>0$ using Proposition \ref{prop2}, for $q > \overline{q}(m_1,\alpha,\epsilon)$, we have:
\begin{align}\label{pp2b}
\Pi_{1,\alpha}(q)\le C_{1} \frac 1{q^{\alpha-1}} \left(1- \frac{C_{2}}{q^{\alpha-1}}\right)^{-1},
\end{align}
where $C_{2}= m_{1}^{\alpha-1}\frac{(\alpha-1)^{\alpha-1}}{\alpha^{\alpha-1}} (\alpha+\epsilon)$.
\end{itemize}
Now if $\alpha\in (2,\infty)$, choose $Q^{*}=\max( \underline{q}(m_1,m_{\alpha},\alpha), \overline{q}(m_1,\alpha))$ and if   $\alpha\in (1,2)$, then fix $\epsilon>0$ and choose $Q^{*}=\max( \underline{q}(m_1,m_{\alpha},\alpha), \overline{q}(m_1,\alpha,\epsilon))$.
 Hence combining \eqref{pp1}, \eqref{pp2a} and  \eqref{pp2b}, for $q>Q^{*}$ we get:
\begin{align*}
q^{-(\alpha-1)} \left(1-\frac{C_{2}}{t^{\alpha-1}}\right) \le \frac{\Pi_{1,\alpha}(tq)}{\Pi_{1,\alpha}(t)} \le q^{-(\alpha-1)} \left(1-\frac{C_{2}}{(tq)^{\alpha-1}}\right)^{-1}.
\end{align*}
Since $1-C_{2}/(tq)^{\alpha-1}\to1$ and $1-C_{2}/t^{\alpha-1} \to 1$,  as $t\to\infty$, we can infer that:
\[ \lim_{t\to\infty} \frac{\Pi_{1,\alpha}(tq)}{\Pi_{1,\alpha}(t)} = q^{-(\alpha-1)}.\]
Hence $\Pi_{1,\alpha}(q)=\sup_{{F} \in \mathcal{F}_{1,\alpha}}\  \mathbb{E}_{{F}}[\tilde{d}-q]_+ \in \RV_{-(\alpha-1)}.$}

\noindent \textit{(b)} As a consequence of Theorem 2.1 and Section 3.1, page 32 in Shapiro and Kleywegt \cite{shapiro02}, we observe that given the ambiguity set $\mathcal{F}_{1,\alpha}$, there exists a non-negative random variable $\tilde{d}^{*}$ following a distribution $F^{*}$ such that, for any $q\ge0$,
\begin{align*}
\Pi_{1,\alpha}(q)= \sup_{{F} \in \mathcal{F}_{1,\alpha}}\  \mathbb{E}_{{F}}[\tilde{d}-q]_+ = \mathbb{E}_{F^{*}}[\tilde{d}^{*}-q]_{+} .
\end{align*}
We can write:
\begin{align}
\mathbb{E}_{F^{*}}[\tilde{d}^{*}-q]_{+} & = \int_{q}^{\infty} \Pr(\tilde{d}^{*}>w)\, \mathrm dw = \int_q^{\infty} \overline{F}^{*}(w)\, \mathrm dw, \label{eq:exprv}
\end{align}
where $\overline{F}^{*}=1-F^{*}$. From part (a), we have for $q\to \infty$,
\[\int_q^{\infty} \overline{F}^{*}(w)\, \mathrm dw  = \mathbb{E}_{F^{*}}[d^{*}-q]_{+} = \Pi_{1,\alpha}(q) \in \RV_{-(\alpha-1)}.\]
Now since $-(\alpha-1)<0$ and $\overline{F}^{*}$ is non-increasing, using Theorem \ref{thm:karamatainv} (b) (the converse part of Karamata's Theorem), we have $\overline{F}^{*} \in \RV_{-{\alpha}}$. Note that for any $\alpha_1\ge 0$, and some $C>0$, we have:
\begin{align*}
\mathbb{E}_{F^{*}}[(\tilde{d}^{*})^{\alpha_1}] = \int_{0}^{C} t^{\alpha_1-1} \overline{F}^{*}(t) \, \mathrm d t + \int_{C}^{\infty} t^{\alpha_1-1} \overline{F}^{*}(t) \, \mathrm d t.
\end{align*}
The first sum in the summand is bounded above by $C^{\alpha_1}$ which is finite. The integrand in the second term $t^{\alpha_1-1} \overline{F}^{*}(t) \in \RV_{\alpha_2}$ where $\alpha_2=-(\alpha-\alpha_1)-1$. For $\alpha_1<\alpha$, we have $\alpha_2<-1$ and  using Theorem \ref{thm:karamata}, $\int_{C}^{\infty} t^{\alpha_1-1} \overline{F}^{*}(t) \, \mathrm d t $ is finite (which is what we need) and regularly varying with index $(\alpha-\alpha_1)$. Hence for $\alpha_1<\alpha$, we have $\mathbb{E}_{F^{*}}[(\tilde{d}^{*})^{\alpha_1}]<\infty$.
Finally, we show that $\mathbb{E}_{F^{*}}[(\tilde{d}^{*})^{\alpha}] = \infty$ which implies that any higher moment will also be infinite.
Note that for any $q>0$ we have
\begin{align*}
\Pi_{1,\alpha}(q) - \Pi_{1,\alpha}(2q) & = \bE_{F^{*}}[\tilde{d}^{*}-q]_{+} - \bE_{F^{*}}[\tilde{d}^{*}-2q]_{+},\\
		   & =  \int_{q}^{2q} \overline{F}^{*} (y)\, \mathrm dy,\\
		   & \le q\oFs(q),
\end{align*}
since $\oFs$ is non-increasing. Hence, for large enough $q$ satisfying both  \eqref{pp1} and \DAS{ \eqref{pp2a}  (or \eqref{pp2b} depending on the value of $\alpha$)},  we have
\begin{align*}
\oFs(q) &\ge \frac1q\left[\Pi_{1,\alpha}(q)-\Pi_{1,\alpha}(2q)\right],\\
            & \ge \frac1q \left[\frac{C_{1}}{q^{\alpha-1}} - \frac{C_{1}}{(2q)^{\alpha-1}} \left(1-\frac{C_{2}}{(2q)^{\alpha-1}}\right)^{-1}\right],\\
            & \ge  \frac1q \left[\frac{C_{1}}{q^{\alpha-1}} - \frac{C_{1}}{(2q)^{\alpha-1}}\times\left(1-\frac{1}{2^{\alpha-1}}\right)^{-1}\right] \quad\quad \text{(for $q^{\alpha-1}>C_{2}$)},\\
            & = \frac{1}{q^{\alpha}}C_{3},
\end{align*}
where $C_{3}=C_{1}(1-1/(2^{\alpha-1}-1))$. Hence we have for $q$ large enough:
\begin{align*}
\mathbb{E}_{F^{*}}[(\tilde{d}^{*})^{\alpha}] &= \int_{0}^{q} t^{\alpha-1} \overline{F}^{*}(t) \, \mathrm d t + \int_{q}^{\infty} t^{\alpha-1} \overline{F}^{*}(t) \, \mathrm d t,\\
 & \ge  \int_{q}^{\infty} t^{\alpha-1} \overline{F}^{*}(t) \, \mathrm d t,\\
 & \ge  \int_{q}^{\infty} t^{\alpha-1}\frac{C_{3}}{t^{\alpha}}\, \mathrm d t =  C_{3}\int_{q}^{\infty} \frac 1t\, \mathrm d t =\infty.
\end{align*}
Hence for any $\alpha_1\ge \alpha$, we also have $\mathbb{E}_{F^{*}}[(\tilde{d}^{*})^{\alpha_1}] =\infty$.
\end{proof}

\DAS{As a consequence of Theorem \ref{thm1}, we can relate the optimal order quantity of the distributionally robust newsvendor, the optimal worst-case newsvendor cost defined in \eqref{c} and the newly characterized distribution $F^*$, when the critical ratio approaches $1$. While a similar characterization has been previously obtained by researchers in modeling the relationship between Value-at-Risk and Conditional Value-at-Risk in risk management for distributions with regularly varying tails (see Proposition 1 in Hua and Joe \cite{zhuli}), the connection to distributionally robust optimization does not seem to be have been made, to the best of our knowledge.
\begin{prop} \label{newprop}
Consider the ambiguity set $\mathcal{F}_{1,\alpha}$ with $\alpha > 1$. For $\eta \in (0,1)$, let $q_{\eta}^{*}$ be an optimal order quantity to the distributionally robust newsvendor in \eqref{c} and $C_{\eta}^*$ be the optimal cost.
Then $q_{\eta}^{*}$ is also optimal to a standard newsvendor problem with the underlying demand distribution $F^{*}$ described in Theorem \ref{thm1}(b) and satisfies the property:
\begin{align}
q_{\eta}^{*} \sim \frac{\alpha-1}{\alpha}\frac{1}{1-\eta} C_{\eta}^{*}, \quad \text{as} \quad \eta \to1. \label{qcequi}
\end{align}
\end{prop}
\begin{proof}
Note that from Theorem  \ref{thm1}(b) , we have
\begin{align*}
q_{\eta}^{*} &=  \arg\min_{q\in \Re_+} \sup_{F\in \mathcal{F}_{1,\alpha}} \left((1-\eta)q+\mathbb{E}_{F}[\tilde{d}-q]_{+} \right) =   \arg\min_{q\in \Re_+}\left((1-\eta)q + \mathbb{E}_{F^{*}}[\tilde{d}-q]_{+} \right),
\end{align*}
and hence is optimal for  \eqref{c1} with $F\equiv F^{*}$ and $\tilde{d}^{*}\sim F^{*}$. Moreover, we have $1-\eta= \mathbb{P}(\tilde{d}^{*}>q_{\eta}^{*})$, and,
\begin{align*}
C_{\eta}^{*} &=  \min_{q\in \Re_+} \sup_{F\in \mathcal{F}_{1,\alpha}} \left((1-\eta)q + \mathbb{E}_{F}[\tilde{d}-q]_{+} \right),\\
  & =  \min_{q\in \Re_+}  \left((1-\eta)q+\mathbb{E}_{F^{*}}[\tilde{d}^{*}-q]_{+} \right),\\
  & = (1-\eta)q^{*}_{\eta}+\mathbb{E}_{F^{*}}[\tilde{d}^{*}-q_{\eta}^{*}]_{+}.
\end{align*}
Since $ \mathbb{P}(\tilde{d}^{*}>x) = \overline{F}^{*}(x) \in \RV_{-\alpha}$, a direct application of Karamata's theorem (cf. \cite{zhuli}, page 351) yields:
\begin{align*}
\lim_{\eta\uparrow 1} \frac{C_{\eta}^{*}}{(1-\eta)q_{\eta}^{*}} & = \lim_{\eta\uparrow 1} \frac{(1-\eta)q^{*}_{\eta}+\mathbb{E}_{F^{*}}[\tilde{d}^{*}-q_{\eta}^{*}]_{+} }{(1-\eta)q_{\eta}^{*}}, \\
               & = 1 +  \lim_{\eta\uparrow 1} \frac{\int_{q_{\eta}^{*}}^{\infty}  \mathbb{P}(\tilde{d}^{*}>x)\; \mathrm dx}{q_{\eta}^{*} \mathbb{P}(\tilde{d}^{*}>q_{\eta}^{*})}, \\
               &  \displaystyle = 1+ \frac{1}{\alpha-1}, \\
               & \displaystyle =\frac{\alpha}{\alpha-1}.
\end{align*}
\end{proof}
}


{\DAS{
\section{Numerical Examples}
In this section, we provide numerical examples to compare the performance of a classical newsvendor model where the demand is assumed to be known with the distributionally robust newsvendor model. We consider the following three demand distributions that possess different kinds of tail behavior:
\texitem{(a)} Exponential random variable with mean $50$
\texitem{(b)} Lognormal random variable with parameters $m = \log(50/\sqrt{2})$ and $s = \sqrt{\log(2)}$
\texitem{(c)} Pareto random variable with shape parameter $\beta = 1 +\sqrt{2}$ and scale parameter $x_m= 50\sqrt{2}/(1+\sqrt{2})$.

The exponential distribution is light-tailed where all moments of finite order exist, the lognormal distribution is heavy-tailed where all moments of finite order exist ,while the Pareto random variable is a heavy-tailed distributon with finite moments only for $\alpha < \beta$. Among these three distributions, only the Pareto distribution is regularly varying. The parameter of the demand distributions are selected such that for all three distributions, the mean is $50$ and standard deviation is $50$. Hence, Scarf's model would prescribe exactly the same optimal order quantity in all the three cases. On the other hand, since the moments $m_{\alpha}$ are different for these demand distributions when $\alpha$ is not equal to $2$, the order quantities from the distributionally robust newsvendor models would change for other values of $\alpha$.

In the numerical experiments, to find the robust optimal order quantities, one approach is to directly use the dual SDP formulations discussed in Section \ref{subsec:3.3}. For example, for $\alpha = 3$, this would reduce to solving:
\begin{equation} \label{dualf3a}
\begin{array}{rlll}
\displaystyle \min_{q,y_0,y_1,y_3,a_1,b_1,c_1,a_2,b_2,c_2} &(1-\eta)q+y_0 + y_1 m_1 + y_{3} m_{3} &\\
\textrm{s.t.} & \begin{bmatrix}
    y_0& 0 &a_1& b_1\\
   0 & y_1-2a_1& -b_1 & c_1 \\
   a_1 & -b_1 & -2c_1 & 0 \\
   b_1& c_1 & 0 & y_3
  \end{bmatrix}\succeq 0 ,  & \\
& \begin{bmatrix}
    y_0+q& 0 &a_2 & b_2\\
   0 & y_1-1-2a_2&-b_2 & c_2 \\
   a_2 & -b_2 & -2c_2 & 0 \\
   b_2& c_2 & 0 & y_3
  \end{bmatrix}\succeq 0 , &\\
  & q \geq 0,
\end{array}
\end{equation}
However, a standard reformulation comes at the price that for large values of $\alpha$, the SDP involves several additional variables, besides $q, y_0, y_1, y_{\alpha}$. In our setting, since the dual constraints are equivalent to nonnegativity constraints of sparse univariate polynomials, we can use a technique from relative entropy reformulations for signomial optimization which preserves sparsity (see Chandrasekaran and Shah \cite{chandrashekaran}). Specifically for any $\alpha > 1$, the problem is given as follows:
\begin{equation} \label{dualf111}
\begin{array}{rlll}
\min & (1-\eta)q + y_0 + y_1 m_1 + y_{\alpha} m_{\alpha} &\\
\textrm{s.t.} & y_0 + y_{\alpha} w^{\alpha} \geq  - y_1 w,  & \forall w \geq 0,\\
& y_0+q + y_{\alpha} w^{\alpha} \geq  (1- y_1) w, & \forall w \geq 0 \\
& q \geq 0,
\end{array}
\end{equation}
where the variables $y_0$ and $y_{\alpha}$ must be nonnegative for feasibility. Using a change of variable with $w = e^z$ where $z \in \Re$ and dividing the first two constraints throughout by $e^z$, we get:
\begin{equation} \label{dualf111a}
\begin{array}{rlll}
\min & (1-\eta)q + y_0 + y_1 m_1 + y_{\alpha} m_{\alpha} &\\
\textrm{s.t.} & y_0  e^{-z} + y_{\alpha} e^{(\alpha-1) z} \geq  - y_1,  & \forall z ,\\
& (y_0+q) e^{-z} + y_{\alpha} e^{(\alpha-1)z} \geq  1- y_1, & \forall z. \\
& q \geq 0,
\end{array}
\end{equation}
Equivalently, the formulation reduces to:
\begin{equation} \label{dualf111b}
\begin{array}{rlll}
\min & (1-\eta)q + y_0 + y_1 m_1 + y_{\alpha} m_{\alpha} &\\
\textrm{s.t.} & \displaystyle \min_{z} y_0  e^{-z} + y_{\alpha} e^{(\alpha-1) z} \geq  - y_1,  & \\
& \displaystyle \min_{z} (y_0+q) e^{-z} + y_{\alpha} e^{(\alpha-1)z} \geq  1- y_1, &  \\
& q \geq 0,
\end{array}
\end{equation}
where the minimization problems over $z$ are convex optimization problems, since the coefficients of $e^{-z}$ and $e^{(\alpha-1)z}$. Now using Lagrangian duality (see Lemma 1 on page 1150 in \cite{chandrashekaran}), we can rewrite the problem as a relative entropy optimization problem:
\begin{equation} \label{dualf111c}
\begin{array}{rlll}
\min_{q,y_0,y_1,y_{\alpha},v_1,v_2,v_3,v_4} & (1-\eta)q + y_0 + y_1 m_1 + y_{\alpha} m_{\alpha} &\\
\textrm{s.t.} & \displaystyle v_1 \log\left(\frac{v_1}{e y_0}\right) + v_2\log\left(\frac{v_2}{e y_{\alpha}}\right) \leq  y_1,  & \\
& \displaystyle  v_3 \log\left(\frac{v_3}{e (y_0+q)}\right) + v_4\log\left(\frac{v_4}{e y_{\alpha}}\right)  \leq  y_1-1, &  \\
& (\alpha-1)v_2 = v_1,\\
& (\alpha-1)v_4 = v_3,\\
& q, v_1, v_2, v_3, v_4 \geq 0,
\end{array}
\end{equation}
which is a convex optimization problem in the variables $q,y_0,y_1,y_{\alpha},v_1,v_2,v_3,v_4$. The advantage of solving (\ref{dualf111c}) is that the size of the problem formulation does not grow unlike the SDP reformulation, thus exploiting sparsity and is particularly efficient when solving the problem for large values of $\alpha$. Such a relative entropy formulation can be solved using an off the shelf convex optimization solver such as MOSEK.

\subsection{Value of Incorporating Moments Beyond Scarf's Model}
We compute the optimal order quantities for the distributionally robust model assuming the highest order moment is given for $\alpha = 4/3, 3/2, 7/4, 2, 3, 5$ and $8$ respectively in cases (a) and (b), while for case (c), we consider $\alpha = 4/3, 3/2, 7/4$ and $2$ only. Note that for the Pareto random variable in case (c), the moments are finite only for $\alpha < 1+\sqrt{2} \approx 2.4142$. We estimate the optimal order quantities for critical ratios $\eta$ in the range $[0.97,0.99998]$.

In Figures \ref{m1}, \ref{m2} and \ref{m3}, we provide the log-log plots of the distributionally robust optimal order quantities and the optimal order quantities for the exponential, lognormal and Pareto distributions respectively. We observe in all the three figures that as higher order moment information is assumed to be known (higher values of $\alpha$), the robust solution gets closer to the optimal order quantity for the distribution, as the critical ratio gets higher. For the lognormal demand distribution, for the specified range of critical ratios, we observe that the optimal order quantity for $\alpha = 8$ still exceeds $\alpha = 5$, but as the critical ratio increases further, this result is reversed (see Table \ref{m2-supp}). Note that on the other hand, Scarf's model would prescribe the same optimal order quantity for all three cases and does not capture the tail behavior. This clearly indicates the value of having additional moment information in better approximating the tail behavior of the optimal order quantity for a given distribution.

\begin{figure}[!htbp]
\centering
\includegraphics[width=11cm] {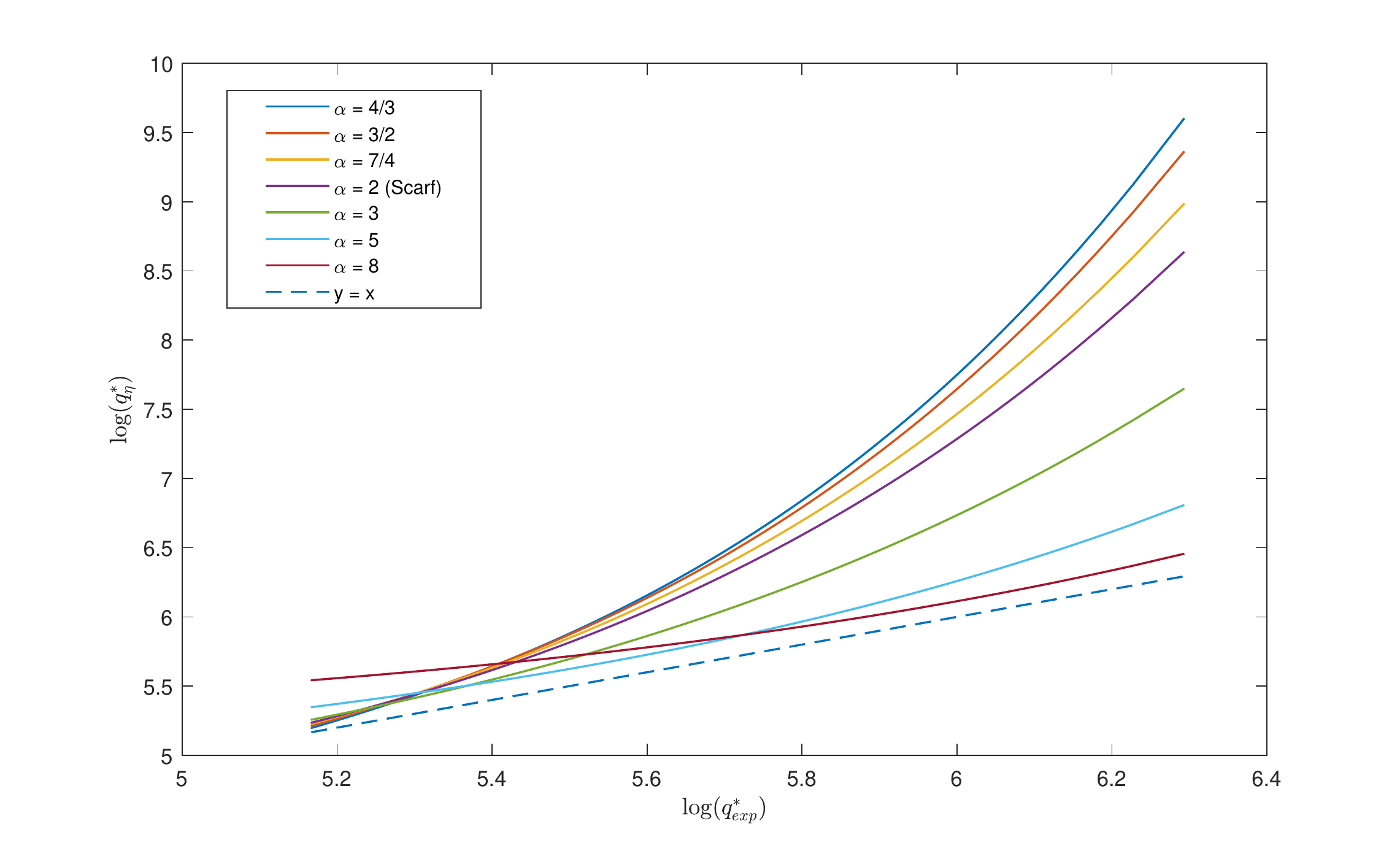}
\caption{The log-log plot compares the optimal order quantities for the distributionally robust newsvendor with $\alpha = 4/3, 3/2, 7/4, 2, 3, 5, 8$  with the optimal order quantity for the exponential distribution for $\eta \in [0.97,0.99998]$. As the figure illustrates for larger critical ratios, the knowledge of higher moment information makes the robust model less conservative and closer to the $y = x$ line.}
\label{m1}
\end{figure}

\begin{figure}[!htbp]
\centering
\includegraphics[width=11cm] {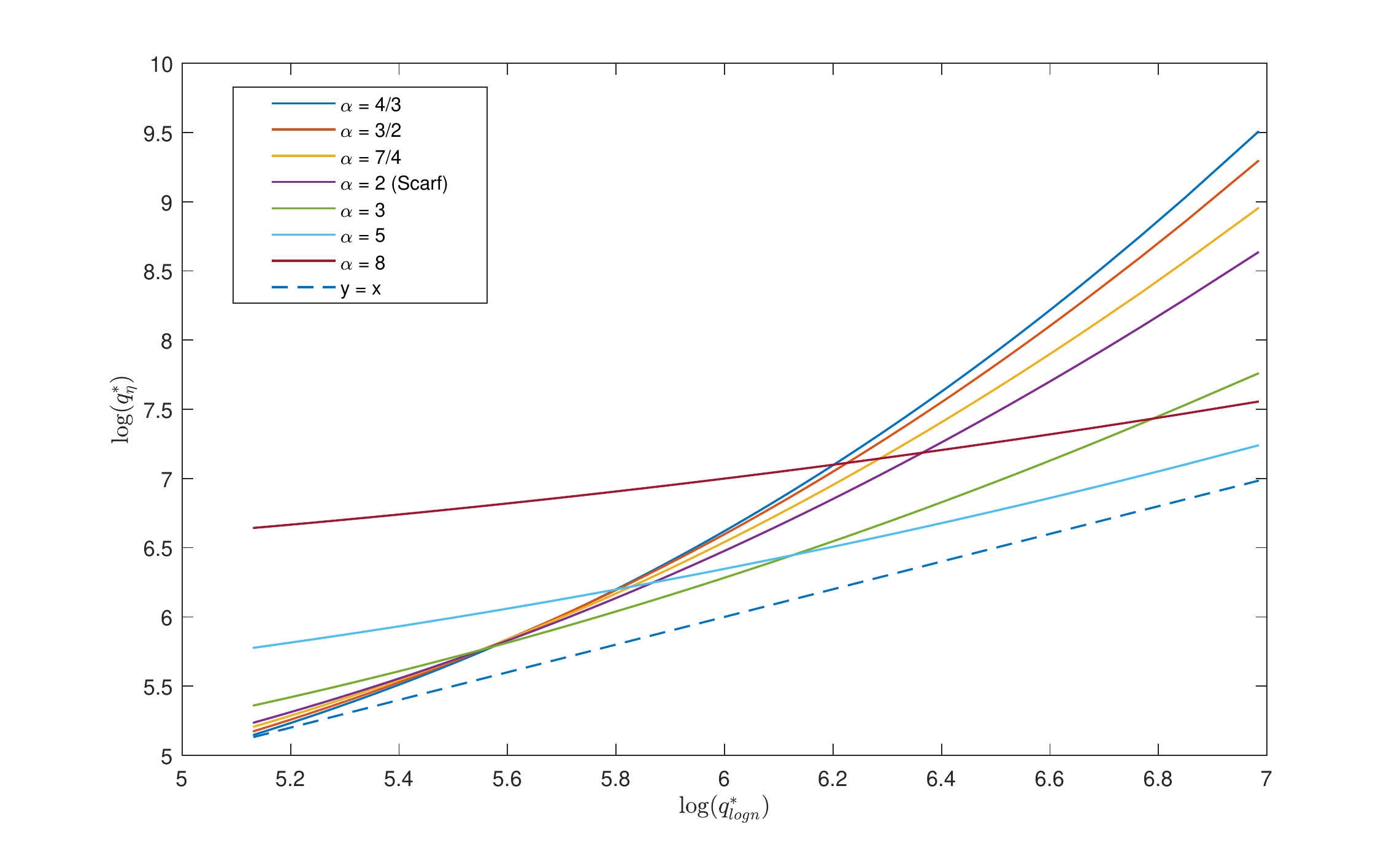}
\caption{The log-log plot compares the optimal order quantities for the distributionally robust newsvendor with $\alpha = 4/3, 3/2, 7/4, 2, 3, 5, 8$  with the optimal order quantity for the lognormal distribution for $\eta \in [0.97,0.99998]$. As the figure illustrates for larger critical ratios, the knowledge of higher moment information makes the robust model less conservative and closer to the $y = x$ line. Only for $\alpha = 8$, the line is above the $\alpha = 5$ line for the chosen critical ratios, but the slope indicates that for even higher critical ratios, the robust order quantities for $\alpha = 8$ will get closer to the $y = x$ in comparison to the $\alpha = 5$. This is verified in Table \ref{m2-supp}.}
\label{m2}
\end{figure}

\begin{figure}[!htbp]
\centering
\includegraphics[width=11cm] {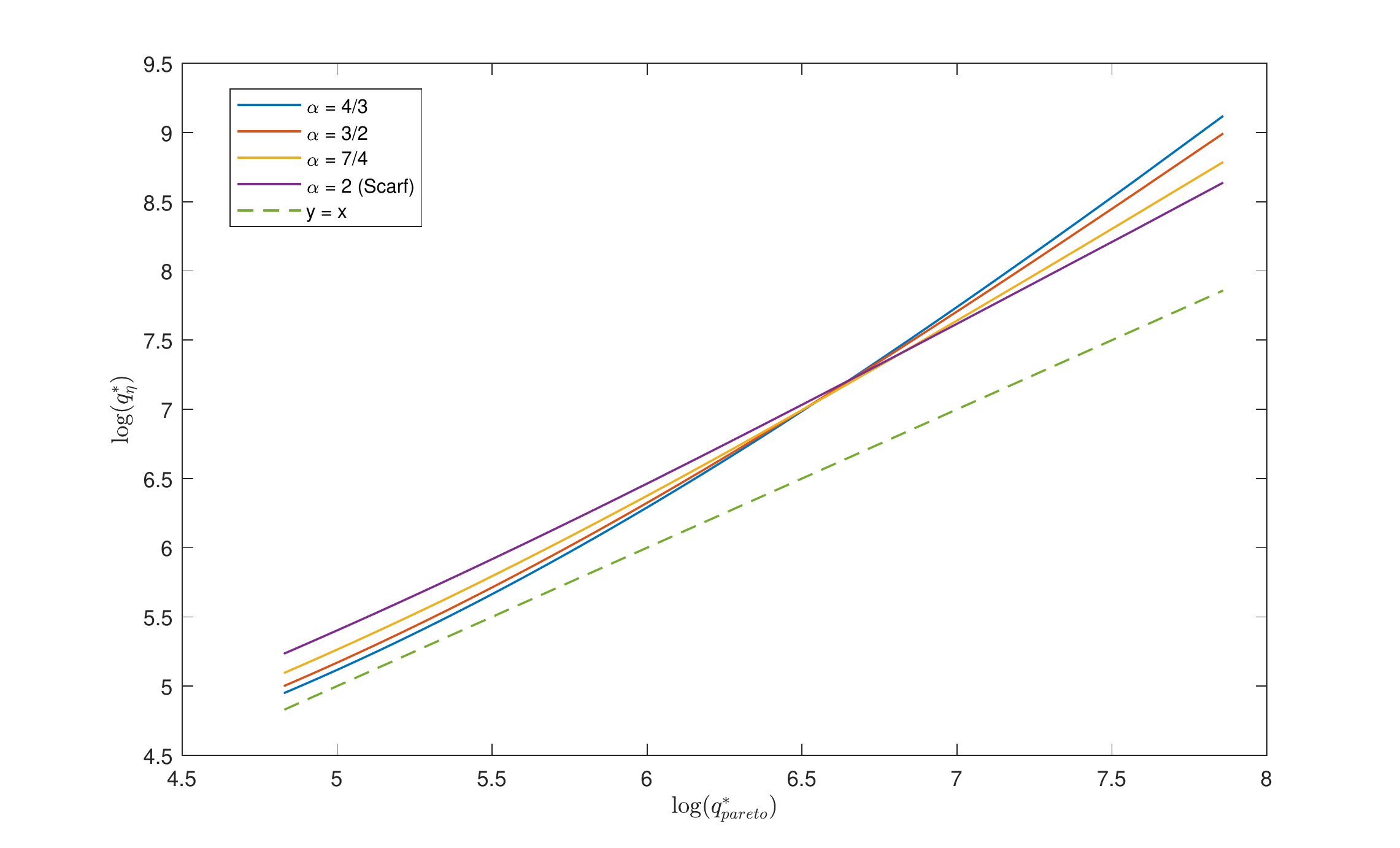}
\caption{The log-log plot compares the optimal order quantities for the distributionally robust newsvendor with $\alpha = 4/3, 3/2, 7/4, 2$  with the optimal order quantity for the Pareto distribution for $\eta \in [0.97,0.99998]$. As the figure illustrates for larger critical ratios, the knowledge of higher moment information makes the robust model less conservative and closer to the $y = x$ line.}
\label{m3}
\end{figure}

\begin{table}[!htbp]
\centering
\begin{tabular}{|c|c|c|c|}
\hline
$\eta$ & $q_{logn}^*$ & $q_{\eta}^*$ $(\alpha = 5)$ & $q_{\eta}^*$ $(\alpha = 8)$ \\
\hline
0.99998 & 1080.46 & 1389.84 & 1913.45\\
0.999998 &  1643.65 & 2177.66 &  2549.77 \\
0.9999998 &   2405.78 &   3371.15 & 3134.68 \\
0.99999998 & 3418.24 & 4414.89 &  4292.02 \\
\hline
\end{tabular}
\caption{Comparison of optimal order quantities for lognormal with the distributionally robust model for $\alpha = 5$ and $\alpha = 8$.}\label{m2-supp}
\end{table}
\subsection{Scaling Behavior of Optimal Order Quantities and Optimal Costs}
We next validate the scaling behavior of the optimal order quantity and the optimal cost for the distributionally robust model as discussed in Proposition \ref{newprop}, illustrating the regularly varying structure and compare it with the corresponding behavior of the optimal solution and the optimal costs for the three distributions in (a)-(c). The scaling constant as the critical ratio approaches $1$ for the robust model is provided in Table \ref{m2-suppa} for the specified values of $\alpha$. In Figures \ref{n1}, \ref{n2} and \ref{n3}, we plot these values for the range of critical ratios in $[0.97,0.99998]$. In the case of the exponential distribution and the lognormal distribution, a simple calculation indicates that these ratios converge to $1$ as the critical ratio approaches $1$, while the distributionally robust newsvendor model shows a different scaling behavior. On the other hand, for the Pareto distribution, the ratio of $(\beta-1)/\beta$ is exactly valid for all critical ratios $\eta$ as shown in the figure.

\begin{table}[!htbp]
\centering
\begin{tabular}{|c|c|c|c|c|c|c|c|}
\hline
$\alpha$ & $4/3$ & $3/2$ & $7/4$ & $2$ & $3$ & $5$ & $8$ \\
\hline
$\displaystyle \lim_{\eta \rightarrow 1} \frac{(1-\eta)q_{\eta}^*}{C_{\eta}^*}$ = $\displaystyle \frac{\alpha-1}{\alpha}$  & $1/4$ & $1/3$ & $3/7$ & $1/2$ & $2/3$ & $4/5$ & $7/8$ \\
\hline
\end{tabular}
\caption{Scaling behavior of the optimal order quantity and the optimal cost for the distributionally robust model.}\label{m2-suppa}
\end{table}

\begin{figure}[htbp]
\centering
\includegraphics[width=11cm] {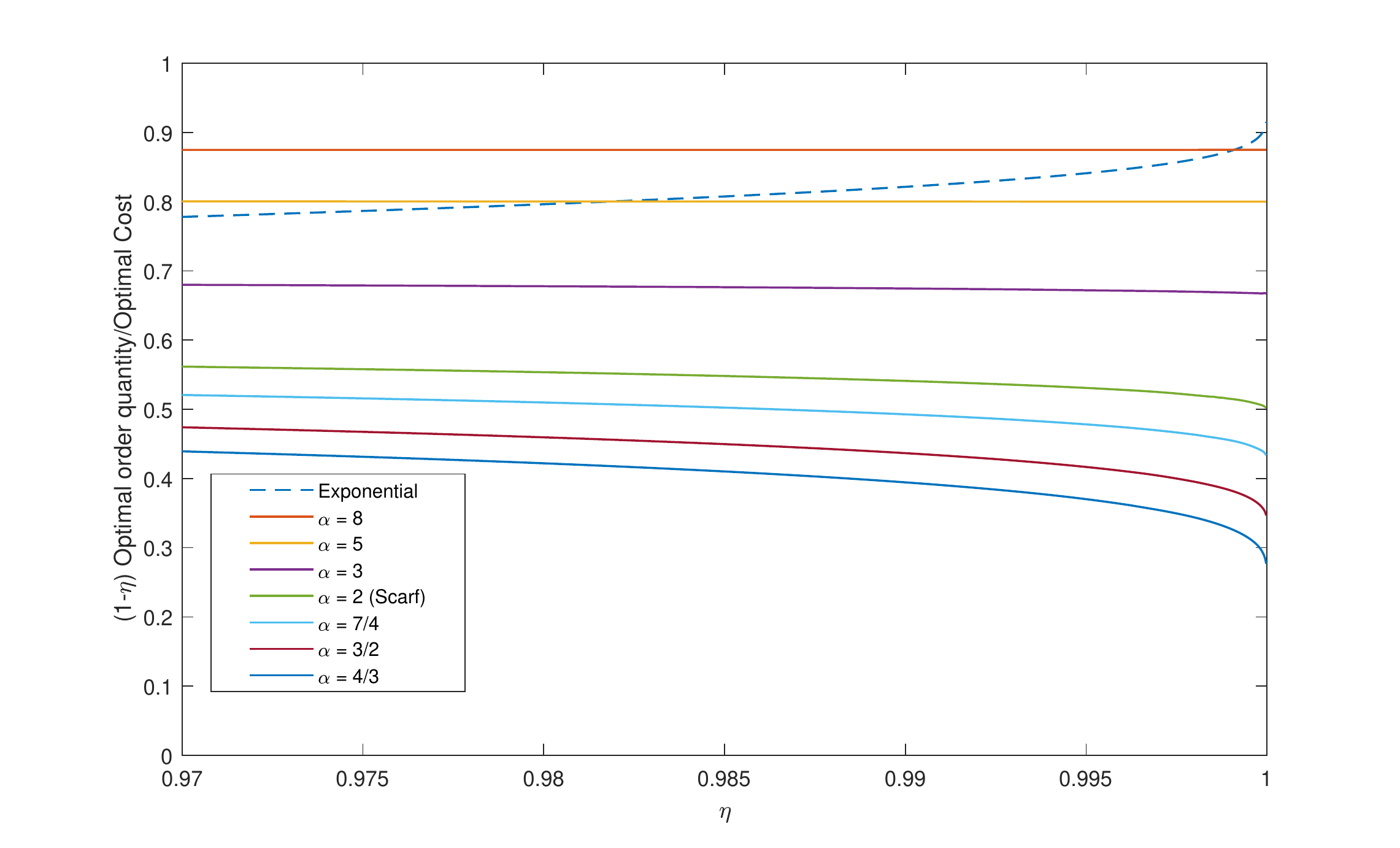}
\caption{The plot provides the ratios $(1-\eta)q_{\eta}^*/C_{\eta}^*$ for the distributionally robust newsvendor model (with different values of $\alpha$) and the corresponding values for the exponential distribution (which is the dashed line, which tends to 1 as $\eta$ tends to 1). The figure illustrates the difference in the scaling behavior of the two models with the limit value given by the numbers in Table \ref{m2-suppa}.} \label{n1}
\end{figure}

\begin{figure}[htbp]
\centering
\includegraphics[width=11cm] {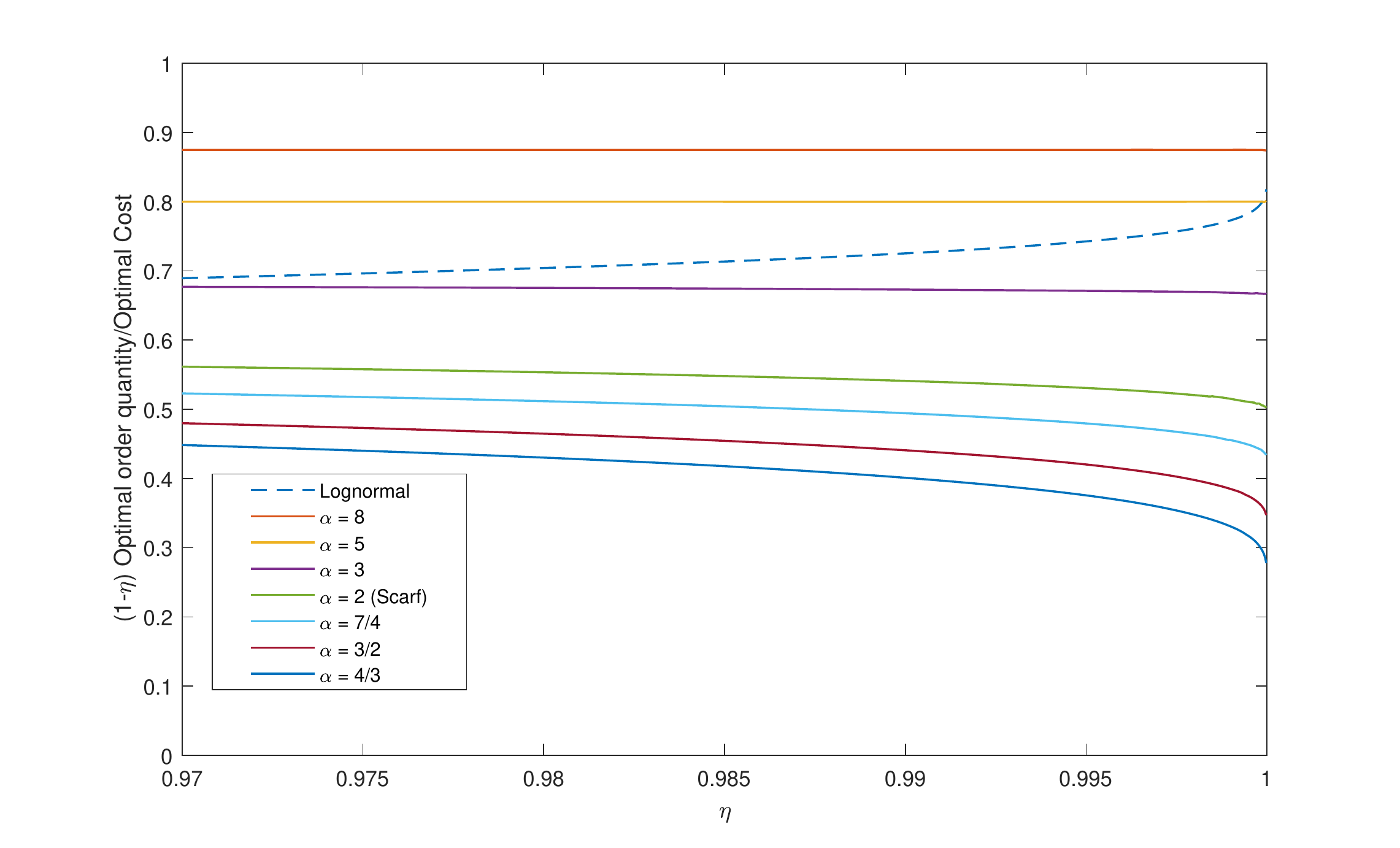}
\caption{The plot provides the ratios $(1-\eta)q_{\eta}^*/C_{\eta}^*$ for the distributionally robust newsvendor model (with different values of $\alpha$) and the corresponding values for the lognormal distribution (which is the dashed line which tends to 1 as $\eta$ tends to 1). The figure illustrates the difference in the scaling behavior of the two models with the limit value given by the numbers in Table \ref{m2-suppa}.}
\label{n2}
\end{figure}

\begin{figure}[htbp]
\centering
\includegraphics[width=11cm] {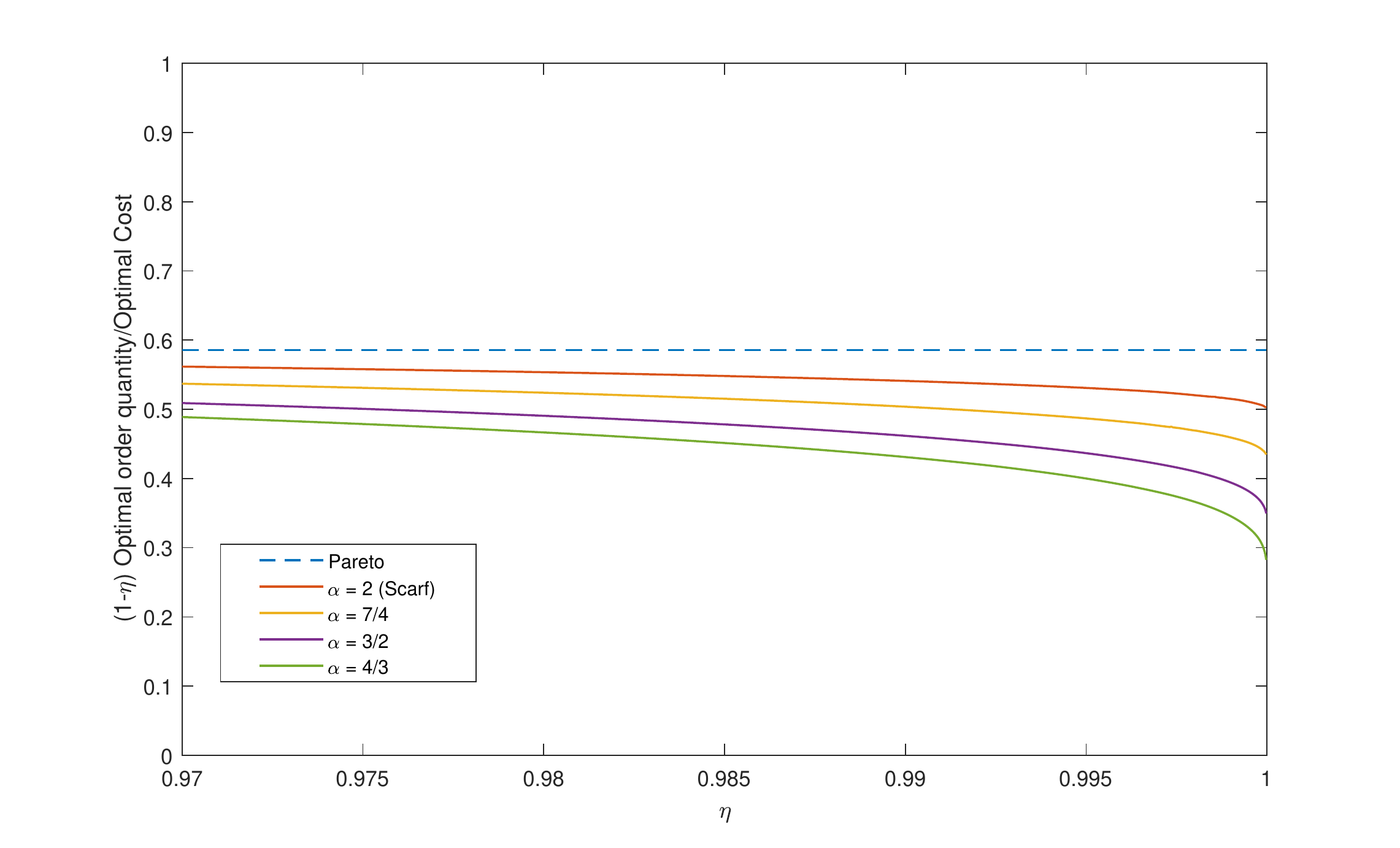}
\caption{The plot provides the ratios $(1-\eta)q_{\eta}^*/C_{\eta}^*$ for the distributionally robust newsvendor model  (with different values of $\alpha$) where the moments are obtained from a Pareto distribution. The Pareto distribution is regularly varying and in this case, the ratio is exactly $\sqrt{2}/(1+\sqrt{2}) \approx 0.5858$.}
\label{n3}
\end{figure}

\subsection{Robustness to Contamination}
In this section, we compare the performance of the two optimal order quantities - the solution to the classical newsvendor problem and the distributionally robust newsvendor problem. To compare the performance, we use a mixture of two distributions, $F_{0}$ and $F_1$, where $F_0$ is the original distribution and $F_{1}$ is a heavy-tailed (contamination) distribution. The mixture distribution is parameterized by $\lambda \in [0,1]$:
\[ F_{\lambda} \equiv (1-\lambda)F_{0} + \lambda F_{1}, \quad 0\le \lambda \le 1.\]
Such a contamination technique has been proposed in Dupa$\check{c}$ov$\acute{a}$ \cite{dupacova} to analyze the stability of optimal solutions in stochastic programs when the true distribution is contaminated by another distribution. For the choice of the distribution $F_0$, we use each of the distributions in (a)-(c). A natural choice for the contamination distribution is $F_1 = F^*$ where $F^*$ is the distribution described in Theorem \ref{thm1}(b) and Proposition \ref{newprop} for a chosen $\alpha$. When $\lambda = 0$, the optimal order quantity is the solution to the newsvendor problem with the corresponding distribution in (a)-(c). On the other hand, when $\lambda = 1$, the optimal order quantity is the solution to the distributionally robust newsvendor for the given $\alpha$. Denote the corresponding optimal order quantities by $q_{0}^*$ and $q_1^*$ respectively. The order quantities satisfy:
\begin{equation} \label{contamination0}
\begin{array}{rlll}
\displaystyle C_0(q_0^*)  & \leq  & \displaystyle C_0(q_1^*), \\
\displaystyle C_1(q_0^*)  & \geq  & \displaystyle C_1(q_1^*), \\
\end{array}
\end{equation}
where $C_0$ is the newsvendor cost under distribution $F_0$ and $C_1$ is the newsvendor cost under distribution $F_1$. Then, a natural question is the what is the value of $\lambda^*$, beyond which the robust order quantity outperforms the classical solution under contamination. If $\lambda^* < 0.5$, it indicates that with less than 50\% contamination, the robust order quantity outperforms the standard newsvendor order quantity. On the other if $\lambda^* > 0.5$, this indicates that more than 50\% contamination is needed for the robust order quantity to outperform the standard newsvendor order quantity. In Figures \ref{o1}, \ref{o2} and \ref{o3}, we plot the $\lambda^*$ values for each of the exponential, lognormal and Pareto distributions with the contaminating distribution given by the regularly varying distribution for the corresponding $\alpha$ value. The figures illustrate that for high service levels, with even a small amount of contamination, the distributionally robust models will outperform the standard newsvendor solution.

\begin{figure}[htbp]
\centering
\includegraphics[width=15cm] {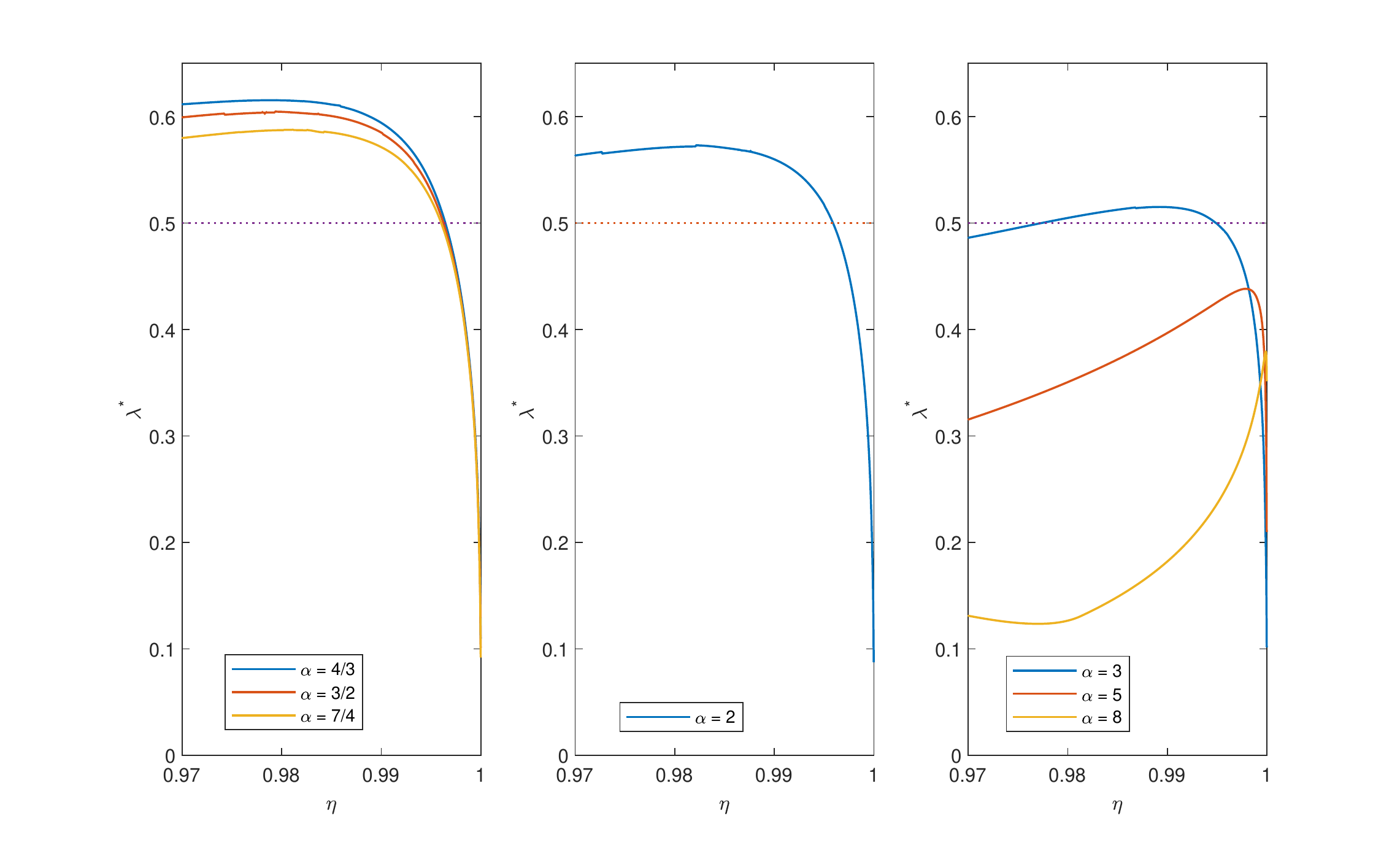}
\caption{The plot provide the $\lambda^*$ values for the case where the original distribution is exponential. In each of the cases, we see that as the critical ratio approaches 1, the value of $\lambda^*$ rapidly drops to 0. This indicates that a small amount of contamination is sufficient for the robust solution to outperform the classical solution for high service levels.}
\label{o1}
\end{figure}

\begin{figure}[htbp]
\centering
\includegraphics[width=15cm] {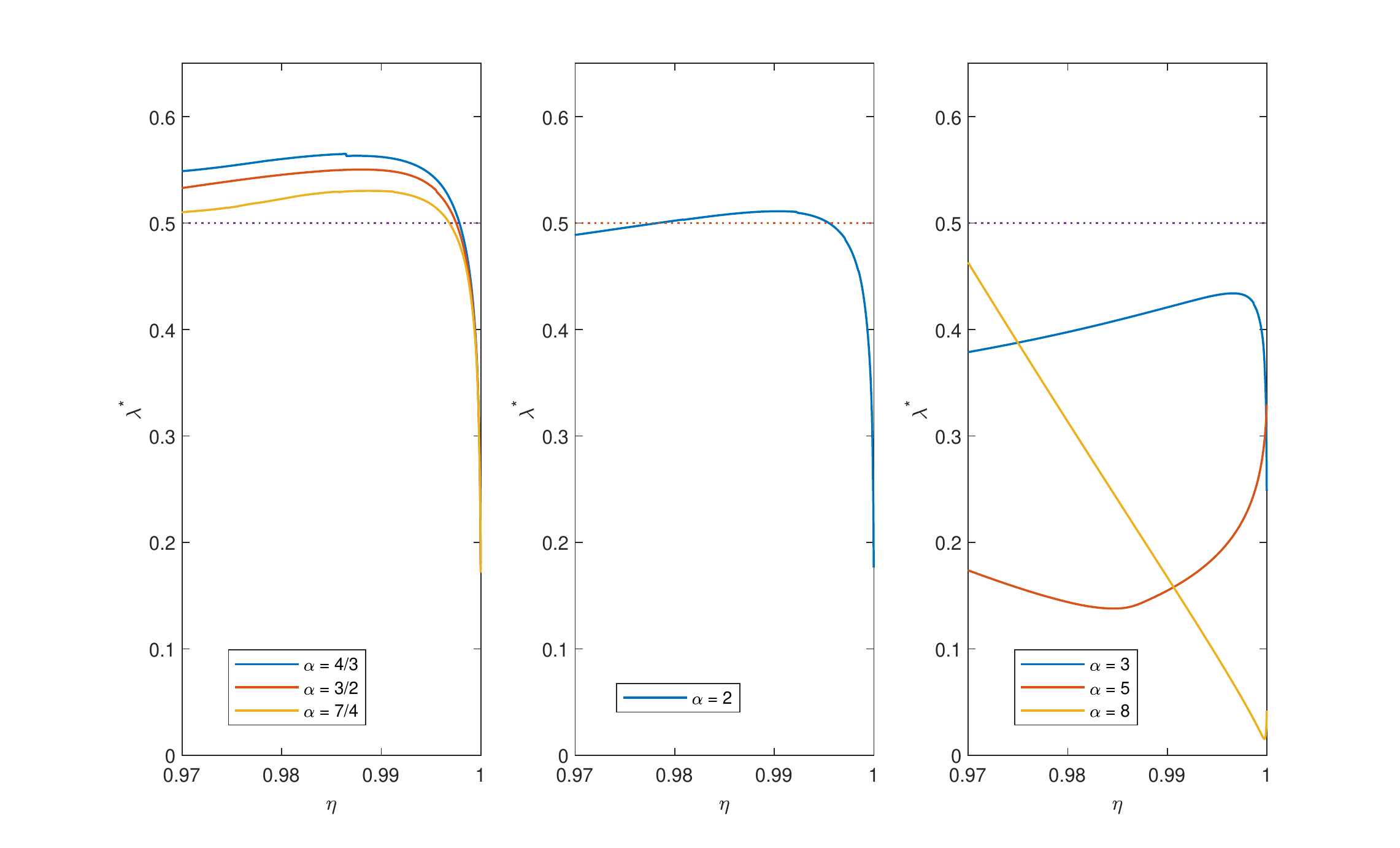}
\caption{The plot provide the $\lambda^*$ values for the case where the original distribution is lognormal. In each of the cases, we see that as the critical ratio approaches 1, the value of $\lambda^*$ rapidly drops to 0. This indicates that a small amount of contamination is sufficient for the robust solution to outperform the classical solution for high service levels.}
\label{o2}
\end{figure}

\begin{figure}[htbp]
\centering
\includegraphics[width=15cm] {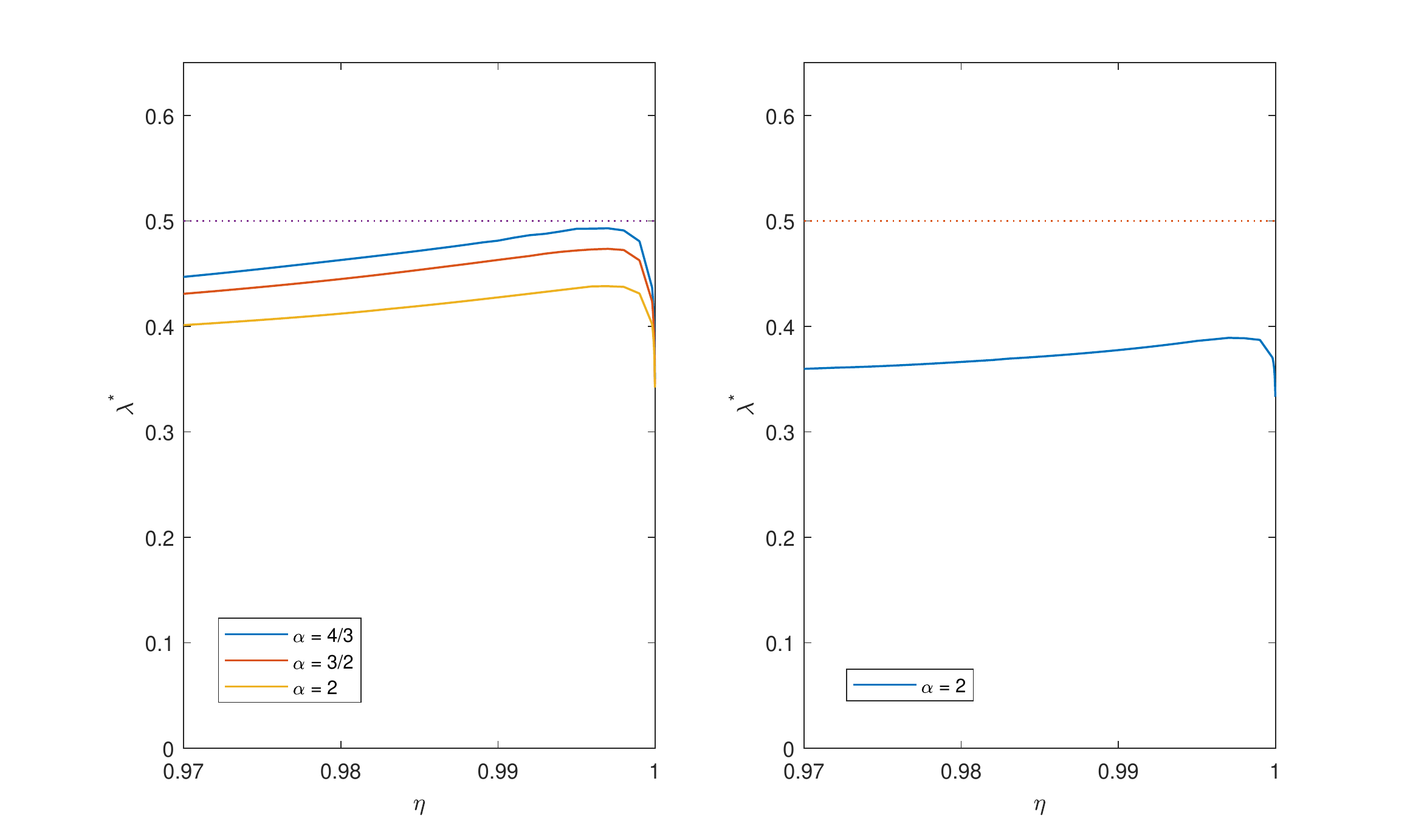}
\caption{The plot provide the $\lambda^*$ values for the case where the original distribution is Pareto. In each of the cases, we see that as the critical ratio approaches 1, the value of $\lambda^*$ rapidly drops to 0. This indicates that a small amount of contamination is sufficient for the robust solution to outperform the classical solution for high service levels.}
\label{o3}
\end{figure}

}}

}

 \section{Conclusion}
 The goal of this paper was to characterize properties of the optimal order quantities in a newsvendor model under a robust framework of distributional ambiguity with moment constraints.  Building on the observation that the optimal order quantity in Scarf's model is also optimal for a censored student-t distribution with parameter 2, we show that by assuming knowledge of the first and $\alpha$-th moment, the optimal order quantity is also optimal for a regularly varying with tail index $\alpha$. This provides a characterization of a new distribution, which does not lie in the original ambiguity set, but for which the order quantity from a robust model, continues to remain optimal. We provied numerical evidence to illustrate these results and its applicability.

Several interesting questions still remain to be answered. While our results provide a characterization of the distribution $F^*$ in the tails with moment information, it would be interesting to see if there is a more precise analytical characterization of other aspects of the distribution $F^*$. Secondly, it would be interesting to characterize the distribution $F^*$ for other types of ambiguity sets that include distributions around a nominal distribution using other probability metrics. We believe this will help managers better understand as to the types of data under which, the solutions from other robust models will do. Lastly, implications of these results for multidimensional newsvendor problems need to be studied. We leave this for future research.



\section*{Acknowledgements}
\textcolor{black}{The authors would like to thank the Area Editor Dan Adelman, the Associate Editor and two anonymous reviewers for their very useful inputs and suggestions on improving the paper.} The authors would also like to thank Guillermo Gallego (HKUST) for providing useful references on this topic and participants at the Distributionally Robust Optimization Workshop at Banff, 2018 for their valuable inputs on this research.

\renewcommand{\baselinestretch}{1.00}

\end{document}